\documentclass[reqno,a4paper,12pt]{amsart}
\usepackage{amsmath,amsthm,amsfonts,amssymb, mathrsfs}
\usepackage{verbatim, graphicx, ifthen, enumitem}
\usepackage[T1]{fontenc}
\usepackage [applemac] {inputenc}
\usepackage{color, hyperref}

\allowdisplaybreaks

\let\oldmarginpar\marginpar
\renewcommand\marginpar[1]{\-\oldmarginpar[\raggedleft\footnotesize #1]%
{\raggedright\footnotesize #1}}
\newtheorem*{thma}{Theorem A}
\newtheorem*{thmb}{Theorem B}
\newtheorem{theorem}{Theorem}[section]
\newtheorem*{theorem*}{Theorem}

\newtheorem{lemma}[theorem]{Lemma}

\newtheorem{proposition}[theorem]{Proposition}
\newtheorem{corollary}[theorem]{Corollary}
\theoremstyle{definition}
\newtheorem{definition}[theorem]{Definition}

\newtheorem{remark}[theorem]{Remark}

\newcommand{\Z}{\mathbb{Z}}

\newcommand{\abs}[1]{|#1|}
\newcommand{\Bigabs}[1]{\Big|#1\Big|}
\newcommand{\Abs}[1]{\left|#1\right|}

\newcommand{\N}{\mathbb{d}}
\newcommand{\R}{\mathbb{R}}

\newcommand{\C}{\mathbb{C}}

\newcommand{\Q}{\mathbb{Q}}

\def\N{\mathbb{N}}
\def\Z{\mathbb{Z}}
\def\Q{\mathbb{Q}}
\def\R{\mathbb{R}}
\def\R{\mathbb{R}}
\def\C{\mathbb{C}}

\def\1{\mathbf{1}}

\newcommand{\dif}{\mathrm{d}}
\newcommand{\e}{\mathrm{e}}
\newcommand{\im}{\mathrm{i}}
\newcommand{\norm}[1]{\|#1\|}

\title{Balian-Low type theorems in finite dimensions}

\author[Nitzan]{Shahaf Nitzan$^1$}
\address{School of Mathematics, Georgia Institute of Technology, Atlanta GA 30332, USA}
\email{shahaf.nitzan@math.gatech.edu}
\thanks{\textcolor{white}{l}$^1$ The first author is supported by NSF grant DMS-1600726.}

\author[Olsen]{Jan-Fredrik Olsen$^2$}
\address{Centre for Mathematical Sciences, Lund University, P.O. Box 118, SE-221 00 Lund, Sweden}
\email{janfreol@maths.lth.se}
\thanks{\textcolor{white}{l}$^2$ The second author is partially  supported by KVA grant MG2016-0023.}


\begin{document}

\maketitle

\begin{abstract}
We formulate and prove finite dimensional analogs for the classical Balian-Low theorem,
and for a quantitative Balian-Low type theorem that, in the case of the real line, we obtained in a previous work. Moreover, we show that these results imply their counter-parts on the real line.
\end{abstract}

\section{Introduction}

\subsection{A Balian--Low type theorem in finite dimensions.}  Let $g\in L^2(\R)$. The Gabor system generated by $g$ with respect to the lattice $\Z^2$ is given by
\begin{equation}\label{gabor}
G(g)=\{e^{2\pi i nt}g(t-m)\}_{(m,n)\in\Z^2}.
\end{equation}
The  classical Balian-Low theorem  \cite{balian1981, battle1988, daubechies1990, low1985} states that if the Gabor system $G(g)$ is an  orthonormal basis, or a Riesz basis, in $L^2(\R)$, then $g$ must have much worse time-frequency localization than what the uncertainty principle permits. The precise formulation is as follows (see \cite{benedetto_heil_walnut1995} for a detailed discussion of the proof and its history).

\begin{thma}[Balian, Battle, Coifman,  Daubechies, Low, Semmes] \label{intro:classical blt} Let $g \in L^2(\R)$. If $G(g)$ is  an orthonormal basis or a Riesz basis in $L^2(\R)$, then
\begin{equation}\label{blt}
\int_{\R}|t|^2|g(t)|^2\dif t=\infty \quad\textrm{or}\quad
\int_{\R}|\xi|^2|\hat{g}(\xi)|^2\dif \xi=\infty.
\end{equation}
\end{thma}

We note that by Parseval's identity,   condition   \eqref{blt}    is equivalent to saying that we must have either
\begin{equation}\label{blt-2}
\int_{\R}|\hat{g}'(\xi)|^2\dif \xi=\infty \quad\textrm{or}\quad
\int_{\R} |{g}'(t)|^2\dif t=\infty.
\end{equation}
That is, these integrals are considered infinite if the corresponding functions are not absolutely continuous, or if they do not have a derivative in $L^2$.

In the last 25 years, the Balian-Low theorem inspired a large body of work in time-frequency analysis, including, among others, a non-symmetric version \cite{benedetto_czaja2006, gautam2008, grochenig1996, nitzan_olsen2011}, an amalagam space version \cite{heil1990}, versions which discuss different types of systems \cite{daubechies_janssen1993, heil_powell2006, heil_powell2009illinois, nitzan_olsen2011},  versions not on lattices \cite{bourgain1988, grochenig_malinnikova2013}, and a quantified version \cite{nitzan_olsen2013}.  The latter result will be discussed in more detail in the second part of this introduction.

Although it provides for an excellent ``rule of thumbs'' in time-frequency analysis, the Balian--Low theorem is not adaptable to many applications since, in realistic situations, information about a signal is given by a finite dimensional vector rather than by a function over the real line. The question of whether a finite dimensional version of this theorem holds has been circling  among researchers in the area\footnote{For more on other uncertainty principles in the finite dimensional setting, we refer the reader to \cite{ghobber_jamming2011} and the references therein.}. In particular, Lammers and Stampe  pose this as the  ``finite dimensional Balian-Low conjecture''  in \cite{lammers_stampe2015}.
Our main goal in this  paper is to answer this question in the affirmative.

Let $N\in\N$ and denote $d=N^2$. We consider the space $\ell_2^d$ of all functions defined over the cyclic group $\Z_d:=\Z/d\Z$ with the normalization,
\begin{equation} \label{normal}
\|b\|_{d}=\frac{1}{N}\sum_{j=0}^{d-1}|b(j)|^2\qquad b=\{b(j)\}_{j=0}^{d-1}.
\end{equation}
To motivate this normalization,  let $g$ be a continuous function in $L^2(\R)$ and put $b(j) =  g({j}/{N})$, $j\in \Z\cap [-N^2/2,N^2/2)$.  That is, the sequence $b \in \ell_2^d$ consists of samples of  the function $g$, at steps of length ${1}/{N}$ over the interval $[-N/2, N/2]$. Then, for ``large enough'' $N$, the above $\ell^2$ norm can be interpreted as a Riemann sum approximating the $L^2(\R)$ norm of $g$. Note that in Section \ref{xcom2}, we define the finite Fourier transform $\mathcal{F}_d$ so that it is   unitary   on $\ell_2^d$.

Let $N\Z_d$ denote the set $\{ N k : k \in \Z \}$ modulo $d$. For $b\in \ell_2^d$, the Gabor system generated by $b$ with respect to  $(N\Z_d)^2$,   is given by
\begin{equation}\label{discrete gabor}
G_d(b):=\{e^{2\pi i \frac{\ell j}{d}}b(j- k)\}_{(k,\ell)\in  (N\Z_d)^2}.
\end{equation}
We point out that, with the choice $b(j) = g(j/N)$,  the discrete Gabor system $G_d(b)$ yields a  discretization of the Gabor system $G(g)$  restricted to the interval $[-N/2,N/2)$.

To formulate the Balian-Low theorem in this setting, we use a discrete version of   condition  \eqref{blt-2}. To this end, we denote the  discrete derivative of a function $b=\{b(j)\}_{j=0}^{d-1}\in \ell_2^d$   by
 \begin{equation*}
\Delta b:=\Big\{b(j+1)-b(j)\Big\}_{j=0}^{d-1},
 \end{equation*}
and put
\begin{equation} \label{laika in space}
	\alpha(N)=\inf \{\|{N}{\Delta b}\|_d^2+\|{N}{\Delta (\mathcal{F}_d b)}\|_d^2\},
\end{equation}
where the infimum is taken over all sequences $b\in \ell_2^d$ for which the system $G_d(b)$ is an orthonormal basis in $\ell_2^d$. We note that for the choice $b(j) = g(j/N)$,  samples of the derivative of $g$ at  the points ${j}/{N}$ are approximated by $N\Delta b$. Therefore, the expression inside of the infimum is a discretization of the integrals in the condition \eqref{blt-2}. Our finite  dimensional version of the Balian-Low theorem,  that answers the finite Balian-Low conjecture in the affirmitive, may now be formulated as follows.

\begin{theorem}\label{finite blt}
  There exist constants $c,C>0$ so that,  for all integers $N \geq 2$, we have
\[
c \log N\leq \alpha(N)\leq  C \log N.
\]
In particular, $\alpha(N)\rightarrow\infty$ as $N$ tends to infinity.

\end{theorem}
\begin{remark}\label{finite blt riesz}
Theorem \ref{finite blt} also holds  in the case that the infimum in $\alpha(N)$ is taken over all $b\in \ell_2^d$ for which the system  $G_d(b)$ is a  basis in $\ell_2^d$ with lower and upper Riesz basis bounds at least $A$ and at most $B$, respectively. In this case, the constants $c,C$ in Theorem \ref{finite blt} depend on $A$ and $B$. (For a precise definition of the Riesz basis bounds see Section \ref{xcom2}). The dependence on the Riesz basis bounds is necessary, in the sense that it can not be replaced by a dependence on the $\ell_2^d$ norm of $b$ (see Remark \ref{homeworld}).
\end{remark}
\begin{remark} The classical Balian-Low theorem  (Theorem A) follows as a corollary of Theorem \ref{finite blt}, as we show in Section \ref{xcom6}.
\end{remark}
\begin{remark}
	By restricting   to $N \geq N_0$, for large $N_0$, the constants in the above theorem improve. Indeed, combining Proposition \ref{prop:alpha and beta} with Remark \ref{unremarkable}, we have
	\[
\lim_{N_0\rightarrow \infty}\Big(\inf_{N\geq N_0}\frac{\alpha(N)}{\log N}\Big)\geq \frac{1}{4\log 2},
\]
	while combining the same proposition with Remark \ref{unremarkabletoo} yields the asymptotic bound
	\[
\lim_{N_0\rightarrow \infty}\Big(\sup_{N\geq N_0}\frac{\alpha(N)}{\log N}\Big)\leq 16\pi^2.
\]
\end{remark}

\subsection{A finite  dimensional quantitative Balian--Low type theorem.}
In \cite{nazarov1993}, F.\,Nazarov obtained the following quantitative version of the classical uncertainty principle: Let $g\in L^2(\R)$ and   $ \mathfrak{Q},\mathfrak{R}\subset\R$ be two sets of finite measure, then
\[
\int_{\R\setminus  \mathfrak{Q}}|g(t)|^2\dif t+ \int_{\R\setminus  \mathfrak{R}}|\hat{g}(\xi)|^2\dif \xi \geq De^{-C| \mathfrak{Q}||\mathfrak{R}|} \|g\|_{L^2(\R)}^2.
\]
In \cite{nitzan_olsen2013}, we obtained the following  quantitative Balian-Low theorem, which is a   modest  analog of  Nazarov's result for generators of Gabor orthonormal bases and,  more generally, Gabor Riesz bases.

\begin{thmb}[Nitzan, Olsen] \label{old quantitative} Let $g \in L^2(\R)$. If $G(g)$ is  an orthonormal basis or a Riesz basis then, for every   $Q,R>1$, we have
 \begin{equation}\label{quant-bltB}
\int_{|t|\geq  Q}|g(t)|^2\dif t+ \int_{|\xi|\geq R}|\hat{g}(\xi)|^2\dif  \xi \geq \frac{C}{ QR},
\end{equation}
where the constant $C>0$ depends only on the Riesz basis bounds of $G(g)$.
\end{thmb}
This quantitative version of the Balian-Low theorem implies the classical Balian-Low theorem  (Theorem A), as well as several extensions of it, including the non-symmetric cases and the amalgam space cases  referred to above. Here, we prove the following  finite dimensional version  of this theorem.

\begin{theorem} \label{finite quantitative BLT 1}
	There exists a constant $C>0$ such that the following holds. Let $N \geq {200}$, and let $b \in \ell_2^d$ (where $d=N^2$) be  such that $G_d(b)$ is an orthonormal basis in $\ell_2^d$.  Then, for all  positive integers $Q, R \in [1, N/{16}]$,  we have	
	\begin{equation*}
		\frac{1}{N}  \sum_{j=N  Q}^{d-1} \abs{b(j)}^2 + \frac{1}{N} \sum_{k= N  R}^{d-1} \abs{\mathcal{F}_d b(k)}^2   \geq \frac{C}{ QR}.
	\end{equation*}
\end{theorem}
\begin{remark}\label{finite quantitative BLT 2}
Theorem \ref{finite quantitative BLT 1}  holds for general  bases as well.   In this case, the constant $C$, as well as the conditions on the sizes of $N$, $Q$ and $R$, depend on the Riesz basis bounds. This more general version of Theorem \ref{finite quantitative BLT 1} is formulated as Theorem
\ref{finite quantitative BLT 1 -rb} in Section 4. (For a precise definition of the Riesz basis bounds see Section \ref{xcom2}).
\end{remark}
\begin{remark} As we show in Section \ref{xcom6}, the quantitative Balian-Low theorem (Theorem B) follows as a corollary of Theorem \ref{finite quantitative BLT 1}.
\end{remark}
\begin{remark} \label{schmabel2}
The conditions appearing in Theorem \ref{finite quantitative BLT 1} are not optimal, but rather, these conditions were chosen to avoid a cumbersome presentation. In particular, we point out that a more delicate estimate in Lemma \ref{rho-lemma}, or the use of a different function, will improve the condition $N \geq 200$. Some modifications in the proof of Lemma \ref{conv-lemma} will improve this condition as well. In addition, a careful analysis of the proof will allow one to improve each one of the conditions $N \geq 200$ and $Q,R \leq N/16$ at the expense of making the constant $C$ smaller. In fact, any two of the previous conditions can be improved at the cost of the third.
\end{remark}

\subsection{Finite  dimension Balian--Low type theorems over rectangular lattices.}  The conclusions   of the classical Balian--Low theorem (Theorem A) and its quantitative version (Theorem B), still hold  if we replace Gabor systems over the square lattice $\Z^2$     by Gabor systems over the rectangular lattices $\lambda\Z\times \frac{1}{\lambda}\Z$, where $\lambda>0$.   Indeed, this is immediately seen  by making an appropriate dilation of the generator function $g$. In the finite dimensional case, however, such dilations are in general not possible. The question of which finite rectangular lattices allow Balian--Low type theorems   therefore has an interest in its own right. We address this in the extensions of of theorems \ref{finite blt} and \ref{finite quantitative BLT 1} formulated below.

Let $M,N\in\N$ and denote $d=MN$. We consider the space $\ell_2^{(M,N)}$ of all functions defined over the cyclic group $\Z_d:=\Z/d\Z$ with   normalization
\begin{equation} \label{normal-rec-intro}
\|b\|^{2}_{(M,N)} =\frac{1}{M}\sum_{j=0}^{d-1}|b(j)|^2\qquad b=\{b(j)\}_{j=0}^{d-1}.
\end{equation}
This non-symmetric normalization is motivated by the fact that if $g$ is a continuous function in $L^2(\R)$ and $b(j) =  g({j}/{  M} )$, $j\in \Z\cap [-MN/2,MN/2)$, then the above $\ell^2$-norm can be interpreted as a Riemann sum for the $L^2(\R)$ norm of $g$ over the interval $[-N/2, N/2]$. Note that in Section \ref{xcom7}, we define the finite Fourier transform $\mathcal{F}_{(M,N)}$ so that it is   a unitary operator from $\ell_2^{(M,N)}$ to $\ell_2^{(N,M)}$.

Let $b\in\ell_2^{(M,N)}$. The Gabor system generated by $b$ with respect to   $M\Z_{d} \times N\Z_{d}$ is given by

\begin{equation}\label{discrete gabor rect}
G_{(M,N)}(b):=\{e^{2\pi i \frac{\ell j}{d}}b(j- k  )\}_{(k,\ell)\in  M\Z_{d} \times  N\Z_{d} }.
\end{equation}

We point out that
making the choice $b(j) = g(j/M)$,  the discrete Gabor system $G_{(M,N)}(b)$ yields a ${1}/{M}$--discretization of the Gabor system $G(g)$ restricted to $[-N/2,N/2]$.

To formulate the discrete Balian--Low theorem in this setting,
we put
\begin{equation} \label{laika in space-rec-intro}
\alpha(M,N)=\inf \{\|{M}{\Delta b}\|_{(M,N)}^{2}+\|{N}{\Delta (\mathcal{F}_{(M,N)} b)}\|_{(N,M)}^{2}\},
\end{equation}
where the infimum is taken over all sequences $b\in \ell_2^{(M,N)}$ for which the system $G_{(M,N)}(b)$ is an orthonormal basis in $\ell_2^{(M,N)}$. We note that for the choice $b(j) = g(j/M)$, the expression inside of the infimum is a discretization of the integrals in   condition \eqref{blt-2}.

We  are now ready to formulate the extension of Theorem \ref{finite blt} to Gabor systems over rectangles.
\begin{theorem}\label{finite blt-rec}
There exist constants $c,C>0$ so that, for all integers $M,N \geq 2$, we have\[
c \log \min\{M,N\}\leq \alpha(M,N)\leq  C \log \min\{M,N\}.
\]
In particular, $\alpha(M,N)\rightarrow\infty$ as $\min\{M,N\}$ tends to infinity.
\end{theorem}

 Similarly, the corresponding extension of Theorem \ref{finite quantitative BLT 1} may be formulated as follows.
\begin{theorem} \label{finite quantitative BLT 1-rec}
	{There exists a constant $C>0$ such that the following holds. Let}
	 $M,N \geq {200}$, {and l}et $b \in \ell_2^{(M,N)}$ be such that $G_{(M,N)}(b)$ is an orthonormal basis in $\ell_2^{(M,N)}$.
	Then, for all  positive integers $Q \leq N/{16}$, $ R \leq M/{16}$, we have	
	\begin{equation*}
		\frac{1}{M}  \sum_{j= M  Q}^{d -1} \abs{b(j)}^2 + \frac{1}{N} \sum_{k= N  R}^{d -1} \abs{\mathcal{F}_{(M,N)} b(k)}^2   \geq \frac{C}{ QR}.
	\end{equation*}

\end{theorem}
We point out that remarks \ref{finite blt riesz} and \ref{finite quantitative BLT 2} also hold for theorems \ref{finite blt-rec} and \ref{finite quantitative BLT 1-rec}, respectively. That is,  these theorems can be extended to generators of general bases with the constants depending only on the Riesz basis bounds. 
Remark \ref{schmabel2} also holds in this case.

%:
\subsection{The structure of the paper.}  In Section \ref{xcom2}, we discuss some preliminaries, in particular the finite and continuous Zak transform and their use in characterizing  orthonormal Gabor bases and Riesz Gabor bases. In Section \ref{xcom3}, we present two improved versions of a lemma we first proved in \cite{nitzan_olsen2013}.  These results  quantify the discontinuity of the argument of a quasi-periodic function. In Section \ref{xcom4}, we apply these lemmas to prove Theorem \ref{finite blt}, while in Section \ref{xcom5}, we use them to prove Theorem \ref{finite quantitative BLT 1}. In Section \ref{xcom6}, we show how the Balian-Low theorem (Theorem A) and its quantitative version (Theorem B) can be obtained from their finite dimensional analogs. Finally, in Section \ref{xcom7}, we discuss theorems \ref{finite blt-rec} and \ref{finite quantitative BLT 1-rec}.  In the most part, we give only a sketch of the proofs for the  rectangular lattice case,  as they are very similar to the proofs we present for the square lattice case.

\section{Preliminaries} \label{xcom2}

\subsection{Basic notations, and the continuous and finite Fourier transforms.}

Throughout the paper, we usually denote by $f$ a function defined over the real line, and by $g$  a function defined over the real line which is a generator of a Gabor system. Similarly, we usually denote by $a$ a discrete function in $\ell^d_2$, and by $b$ a function in $\ell^d_2$ which is a generator of a Gabor system.

For $f \in L^1(\R)$, and with the usual extension to $f \in L^2(\R)$, we use the Fourier transform
\begin{equation*}
	\mathcal{F}f(\xi)  = \int_\R f(t) \e^{-2\pi \im \xi t} \, \dif \xi.
\end{equation*}
We let $\mathcal{S}(\R)$ denote the Schwartz class of functions $\phi$ which are infinitely many times differentiable, and that satisfy $\sup_{t \in \R} \abs{t^k \phi^{(\ell)}(t)} < \infty$ for all $k, \ell \in \N$.

Recall from  the introduction that for $N\in\N$ and $d=N^2$, we denote by $\ell_2^d$ the space of all functions defined over the cyclic group $\Z_d:=\Z/d\Z$ with the normalization 
\[
\|a\|^2_{d}=\frac{1}{N}\sum_{j=0}^{d-1}|a(j)|^2, \qquad a=\{a(j)\}_{j=0}^{d-1}\in \ell_2^d.
\]

For   $a \in \ell_2^d$, we use the finite Fourier transform
\begin{equation*}
%:
	\mathcal{F}_d (a)  (k) =  \frac{1}{N} \sum_{j=0}^{d-1} a(j) \e^{-2\pi \im \frac{jk}{d}}.
\end{equation*}
With the chosen normalization, the finite Fourier transform is unitary on $\ell_2^d$.  We define  the periodic convolution of $a, b \in \C^d$,
  by
\begin{equation*}
	(a \ast b)(k) = \frac{1}{N} \sum_{j=0}^{d-1} a(k-j) b(j),
\end{equation*}
and  note the convolution relation $\mathcal{F}_d(a \ast b) = \mathcal{F}_d (a) \cdot  \mathcal{F}_d (b)$.   Observe that, for the choice $a(j) = f(j/N)$, the discrete Fourier transforms and convolutions yield natural discretizations of their respective counterparts on $\R$.

Also, recall   that for a sequence   $a\in \ell_2^d$, we denote the discrete derivative by $\Delta a(j) = a(j+1) -  a(j)$.
From time to time, we encounter sequences depending on more than one variable, say $a\big(k +  \psi (s)\big)$ where $\psi$ is some function depending on the integer $s$. In this case, we write $\Delta_{(s)}$ if we want to indicate that the difference is to be taken with respect to $s$. That is, $\Delta_{(s)} a\big(k+ \psi (s)\big) = a\big(k +  \psi (s+1)\big) - a\big(k + \psi (s)\big)$.

We will also  consider functions   $W : \Z^2_d \rightarrow \C$  for which we use the notations
\begin{align*}
	\Delta W(m,n) &:=   W(m + 1,n) - W(m,n) \qquad \text{and} \\[2mm]
	\Gamma W(m,n) &:=   W(m,n+1) - W(m,n).
\end{align*}
Finally, we let $\ell_2([0,N-1]^2)$ denote the space of sequences supported on $\Z^2 \cap [0,N-1]^2$ with norm
\begin{equation*}
	\frac{1}{d} \sum_{m,n=0}^{N-1} \abs{W(m,n)}^2,
\end{equation*}
where, as usual, $d=N^2$. We note that this normalization can be related to the process of sampling an $L^2$
function over $[0, 1]^2$, on the vertices of  squares of side length $1/N$, and computing the corresponding
Rieman sum.

\subsection{The continuous and finite Zak transforms}
On $L^2(\R)$, the Zak transform is defined as follows.  
\begin{definition} Let $f \in L^2(\R)$. The \emph{continuous Zak transform} of $f$ is given by
	\begin{equation*}
		Zf(x,y) = \sum_{k\in \Z} f(x-k)\e^{2 \pi \im ky}, \quad (x,y) \in \R^2.
	\end{equation*}
\end{definition}
We summarise the basic properties of the continuous Zak transform in the following lemma. Proofs for these properties, as well as further discussion of the Zak transform, can be found, e.g., in \cite[Chapter 8]{grochenig2001}.
\begin{lemma} \hspace{0mm} \label{ztprop...} Let $f \in L^2(\R)$.  The following hold.
\begin{itemize}
	\item[(i)] The Zak transform is quasi-periodic on $\R^2$ in the sense that
\begin{equation} \label{cond:qp}
	Zf(x+1,y) = \e^{2\pi \im y} Zf(x,y) \quad \text{and} \quad Zf(x,y+1) = Zf(x,y).
\end{equation}
In particular, this means that the function $Zf$ is determined by its values on $[0,1]^2$.
	\item[(ii)] The Zak transform is a unitary operator from $L^2(\R)$ onto $L^2([0,1]^2)$.
	\item[(iii)] The Zak transform and the Fourier  transform satisfy the relation
	\begin{equation*}
		Z (\mathcal{F}f)(x,y) = \e^{2\pi \im xy} Zf(-y,x).
	\end{equation*}
	\item[(iv)]  For  $\phi \in \mathcal{S}(\R)$, the  Zak transform satisfies the convolution relation
\begin{equation*}
	Z (f \ast  \phi) = Z (f) \ast_1 \phi,
\end{equation*}
where the subscript of $\ast_1$ indicates that the convolution is taken with respect to the first variable of the  Zak transform.
\end{itemize}

\end{lemma}

Next, we discuss a  Zak transform for $\ell_2^d$ which appears in, e.g., \cite{auslander_gertner_tolimieri1992}.
\begin{definition} \label{definition:finite zak}
Let $N\in\Z$ and set $d=N^2$. The \emph{finite Zak transform} of $a\in \ell_2^d$,  with respect to $(N\Z_d)^2$, is given by
	\begin{equation*}
		Z_d (a) (m,n)  =  \sum_{j=0}^{N-1}  a(m-Nj) \e^{2\pi \im \frac{jn}{N}}, \qquad (m,n) \in \Z_d^2.
	\end{equation*}
\end{definition}

Note that with this definition, $Z_d(a)$ is well-defined as a function on $\Z_d^2$ (that is, it is $d$-periodic separately in each variable).

The basic properties of the finite Zak transform mirror closely those of the continuous Zak transform and are stated in the following lemma. Parts (i), (ii) and (iii) of this Lemma can be found as theorems 1, 3 and 4 in \cite{auslander_gertner_tolimieri1992}.  Part (iv) follows immediately from the definitions of the Zak transform and the convolution.
\begin{lemma} \hspace{0mm} \label{lemma:basics on finite Zak} Let $N\in \N$, $d=N^2$ and $a \in \ell_2^d$. Then the following hold.
\begin{itemize}
	\item[(i)]  The function $Z_d(a)$   is $N$-quasi-periodic on $\Z_d^2$ in the sense that
\begin{equation} \label{cond:qp2}
\begin{aligned}
	Z_d(a)(m + N,n) &= \e^{2\pi \im \frac{n}{N}}Z_d(a)(m,n),  \\[2mm]
	 Z_d(a)(m,n+N) &=  Z_d (a)\,(m,n).
\end{aligned}
\end{equation}
 In particular,  $Z_d(a)$ is determined by its values on the set $\Z^2 \cap [0,N-1]^2$.

	\item[(ii)] The   transform $Z_d$ is a unitary operator from $\ell_2^d$ onto $\ell_2 ([0,N-1]^2)$.
	\item[(iii)]  The finite Zak transform and the finite Fourier transform satisfy the relation
\begin{equation} \label{Zak and Fourier}
	Z_d (\mathcal{F}_d a)(m,n) =   \, \e^{2\pi \im \frac{mn}{d}} Z_d(a) (-n,m).
\end{equation}
	\item[(iv)] The finite Zak transform satisfies the convolution relation
\begin{equation*}
	Z_d (a \ast  \phi) = Z_d(a) \ast_1 \phi,\qquad a,  \phi \in \ell_2^d.
\end{equation*}
where the subscript of $\ast_1$ indicates that the convolution is taken with respect to the first variable of the finite Zak transform.
\end{itemize}
\end{lemma}

\begin{remark}
 We will make use of a somewhat more general property than the $N$-quasi periodicity. Namely, we will be interested in functions $W : \Z^2_d \rightarrow \C$ satisfying
\begin{equation} \label{cond:qpuptoconstant}
\begin{aligned}
	W(m + N,n) &=  \eta \,\e^{2\pi \im \frac{n}{N}}W(m,n),  \\[2mm]
	 W(m,n+N) &=  W(m,n).
\end{aligned}
\end{equation}
where $\eta$ is a unimodular constant.   In particular, we note that if a function is $N$-quasi-periodic then any translation of it satisfies the relations \eqref{cond:qpuptoconstant}. For easy reference to this property we will call a function satisfying it \textit{$N$-quasi-periodic up to a constant}.
\end{remark}

We will make use of the following lemma,  which is a finite  dimensional analog of inequality (16) from \cite{nitzan_olsen2013}.
 \begin{lemma} \label{convolution inequality 0} Let $N\in\N$ and $d=N^2$. Suppose that $a, \phi \in \ell_2^{d}$ and $k\in \N$, then it holds that
	\begin{equation}
	\begin{aligned}
		\abs{Z_d(a \ast  \phi)(m+k,n)-Z_d(a \ast  \phi)(m,n)}
		&\leq \frac{k}{N} \norm{Z_d (a)\,}_{L^\infty} \sum_{j=0}^{d-1} \abs{\Delta \phi(j)}.
	\end{aligned}
	\end{equation}
\end{lemma}
\begin{proof}
For a function $a\in \ell_2^d$, write $\Delta_ka(n)=a(n+k)-a(n)$.
Property (iv) of Lemma \ref{lemma:basics on finite Zak} implies that $\Delta_k Z_d(a \ast  \phi)=Z_d(a)\ast_1(\Delta_k \phi)$. The result now follows by applying the triangle inequality to $\sum|\Delta_k \phi|$.
\end{proof}

\subsection{Gabor Riesz bases and the Zak transform}
A system of vectors $\{f_n\}$ in a separable Hilbert space $H$ is called a \textit{Riesz basis} if it is the image of an orthonormal basis under a bounded and invertible linear transformation $T:H\mapsto H$. Equivalently,  the system $\{f_n\}$  is a Riesz basis if and only if it is complete in $H$ and satisfies the inequality
\begin{equation}\label{rb}
A\sum| c_n|^2\leq\|\sum  c_nf_n\|^2\leq B\sum| c_n|^2
\end{equation}
for all finite sequences of complex numbers $\{c_n\}$, where $A$ and $B$ are positive constants. The largest $A$ and smallest $B$ for which (\ref{rb}) holds are called the lower and upper Riesz basis bounds, respectively. We note that every basis in a finite dimensional space is a Riesz basis.

The proof for Part (i) of the following proposition can be found, e.g., in \cite[Corollary 8.3.2(b)]{grochenig2001}, while part (ii) can be found in \cite[Theorem 6]{auslander_gertner_tolimieri1992}.
\begin{proposition} \label{gaboriesz}
\begin{itemize}
\item[(i)] Let $g\in L^2(\R)$. Then, $G(g)$ is a Riesz basis in $L^2(\R)$ with Riesz basis bounds $A$ and $B$ if and only if $A\leq |Zg(x,y)|^2\leq B$ for almost every $(x,y)\in[0,1]^2$.
\item[(ii)] Let $N\in \N$, $d=N^2$ and $b\in \ell_2^d$. Then, $G_d(b)$ is a basis in $\ell_2^d$ with Riesz basis bounds $A$ and $B$ if and only if $A\leq |Z_d(b)(m,n)|^2\leq B$ for all $(m,n) \in  [0,N-1]^2 \cap \Z^2$.

\end{itemize}
\end{proposition}

\subsection{Relating continuous and finite signals}

In the introduction, we motivated our choices of normalizations by relating   finite signals to   samples of   continuous ones. In this subsection, we   formulate this relation precisely.

Fix $N \in \mathbb{N}$ and let $d=N^2$. For a function $f$ in the Schwartz class $\mathcal{S}(\R)$, we denote its $N$-periodisation by
\begin{equation*}
	P_N f(t)  = \sum_{\ell= - \infty}^\infty  f(t+ \ell N),
\end{equation*}
and   the $N$-samples of a continuous $N$-periodic function $h$ by
\begin{equation*}
	S_N \, h = \Big\{ h \Big(\frac{j}{N}\Big) \Big\}_{j=0}^{d-1}.
\end{equation*}

First we relate these   operators to the Fourier transform and the Zak transform via Poisson--type formulas. We note that part (i) of the following proposition is stated without proof in  \cite{auslander_gertner_tolimieri1991}.

\begin{proposition}\label{proposition:trump1}
	\hspace{2mm}
For $f$ in the Schwartz class $\mathcal{S}(\R)$,   the following hold.
	\begin{itemize}
\item[(i)]   For every $N\in\N$ and $(m,n)\in\Z^2$, we have
		\begin{equation} \label{potus}
			Z_{d} (S_N {P_N f})(m,n) = Zf(m/N, n/N).
		\end{equation}
		\item[(ii)] For every $N\in\N$, we have
		\begin{equation} \label{eq:*}
			  \mathcal{F}_{d}S_N P_N f = S_N P_N \mathcal{F} f.
		\end{equation}
	\end{itemize}
\end{proposition}

\begin{proof}
	(i):   Let $f \in \mathcal{S}(\R)$. Then
	\begin{align*}
		Z_d(S_N P_N f)(m,n) &= \sum_{j=0}^{N-1} \Big( \sum_{\ell=-\infty}^\infty f(m/N - j + \ell N) \Big) \e^{2\pi \im \frac{jn}{N}} \\[2mm]
		&= \sum_{k=-\infty}^\infty f(m/N -k) \e^{2\pi \im \frac{n}{N}k} \\[2mm]
		&= Zf(m/N,n/N).
	\end{align*}
(ii): Observe that  part (i) holds for both $f$ and   $\mathcal{F}f$. With this, in combination with    parts (iii) of    Lemma \ref{ztprop...} and Lemma  \ref{lemma:basics on finite Zak}, the proof of part (ii) follows.
\end{proof}

\begin{remark} \label{hasta la vista}
\begin{itemize}
\item[(i)]	Although Proposition \ref{proposition:trump1} is   formulated for functions in the Schwartz class $\mathcal{S}(\R)$, it is readily checked that it   holds for all functions $f \in L^2(\R)$ which satisfy both $\sup_{t \in \R} \abs{t^2 f(t)} < \infty$ and $\sup_{\xi \in \R} \abs{\xi^2 \hat f(\xi)} < \infty$.
\item[(ii)]  In Section \ref{seksjon 6}, we obtain a version of Proposition \ref{proposition:trump1} which holds for all $f \in L^2(\R)$.
\end{itemize}
\end{remark}
\begin{comment}
\begin{remark} \label{hasta la vista}
Although Proposition \ref{proposition:trump1} is   formulated for functions in the Schwartz class $\mathcal{S}(\R)$, it is readily checked that it   holds for all functions $f \in L^2(\R)$ which satisfy both $\sup_{t \in \R} \abs{t^2 f(t)} < \infty$ and $\sup_{\xi \in \R} \abs{\xi^2 \hat f(\xi)} < \infty$.
\textcolor{magenta}{Moreover, i}n Section \ref{seksjon 6}, we obtain a version of Proposition \ref{proposition:trump1} which holds for all $f \in L^2(\R)$.
\end{remark}
\end{comment}

We end this section with a lemma that relates the discrete and continuous derivatives.
\begin{lemma}\label{derder}
Let $f\in L^2(\R)$ be a function that satisfies the condition of Remark \ref{hasta la vista}(i), and denote $a=S_NP_Nf$. Then,
\[
\sum_{j=0}^{d-1}|\Delta a_j|  \leq  \int_{\R}|f'(x)|\dif x.
\]
\end{lemma}
\begin{proof}
We have
\[
\begin{aligned}
\sum_{j=0}^{d-1}|\Delta  a_j|&  \leq  \sum_{j=0}^{d-1}\sum_{\ell= - \infty}^\infty \Big |f\Big(\frac{j+1}{N}+ \ell N\Big)-f\Big(\frac{j}{N}+ \ell N\Big)\Big |\\[2mm]
& \leq   \sum_{j=0}^{d-1}\sum_{\ell= - \infty}^\infty \int_{\frac{j}{N}+\ell N}^{\frac{j+1}{N}+\ell N}|f'(x)|\dif x=\int_{\R}|f'(x)|\dif x.
\end{aligned}
\]
\end{proof}

\section{Regularity of the Zak transforms} \label{xcom3}

Essentially, this paper is about the regularity of Zak transforms (or rather, their lack of such). In this section, we formulate a few lemmas in this regard.

\subsection{`Jumps' of quasi-periodic functions on $\Z_d^2$}

It is well known that the argument of a quasi-periodic function on $\R^2$ cannot be continuous (see, e.g., \cite[Lemma 8.4.2]{grochenig2001}). In  \cite[Lemma 1]{nitzan_olsen2013},  we show that such a function has to `jump' on all  rectangular lattices (see Remark \ref{schremark}, below). The latter lemma is finite dimensional in nature.  Below, we formulate it as such, as well as improve the constants used in the lemma.
To this end,  we use the notation
\begin{equation*}
	 \abs{h} > \delta \quad (\mathrm{mod } \; 1)
\end{equation*}
to denote that
 $$\inf_{n \in \Z} \abs{h - n} > \delta.$$

\begin{lemma}\label{lemma:jump 1}
Fix integers $K,L\geq 2$. Let $\gamma\in\C$ and let $H$ be a function on
$([0,K] \times [0,L]) \cap \Z^2$
that satisfies
\begin{align*}
	H(i,L)&=H(i,0) \:\:\:(\mathrm{mod } \, 1) &&   i \in [0,K]\cap \Z \quad \text{and} \\[2mm]
	H(K,j)&= H(0,j)+\gamma+\frac{j}{L} \:\:\:(\mathrm{mod } \, 1)&&    j \in [0,L]\cap \Z.
\end{align*}
 Then, there exist $(i,j)\in  ([0,K-1] \times [0,L-1]) \cap \Z^2$  such that
	\begin{equation} \label{inequality:jump 10}
		\Abs{  \Delta H(i,j)    }  \geq \frac{1}{4} \quad (\mathrm{mod } \, 1)
	\end{equation}
		or
	\begin{equation} \label{inequality:jump 20}
		\Abs{  \Gamma H(i,j)}  \geq \frac{1}{4}  \quad (\mathrm{mod } \, 1).
	\end{equation}
\end{lemma}

\begin{remark} \label{schremark}
	From time to time, we refer to functions satisfying   conditions such as \eqref{inequality:jump 10} or \eqref{inequality:jump 20} as having `jumps'.  This notion of `jumps' is imprecise and merely meant to be descriptive, and will depend on the context of the given situation.
\end{remark}
\begin{remark}
	Note that for $K=L=N$, the argument of an $N$-quasi-periodic up to a constant function (see  (\ref{cond:qpuptoconstant}))
satisfies the conditions of Lemma \ref{lemma:jump 1}.
\end{remark}

\begin{proof}[Proof of Lemma \ref{lemma:jump 1}]
	 To obtain a contradiction, we assume that neither \eqref{inequality:jump 10} nor (\ref{inequality:jump 20}) hold. For the sake of convenience, we choose a specific branch of the function $H$   as follows: First, write $h_{i,j}=H(i,j)$. We choose the row $h_{i,0}$ so that $\abs{\Delta h_{i,0}}<1/4$ for each $i \in[0,K-1] \cap \Z$.
		Next, for each fixed column $i \in  [0,K] \cap \Z$, we choose $h_{i,j}$ so that $\abs{\Gamma h_{i,j}} < 1/4$ for each $j \in [0,L-1] \cap \Z$. With this, $h_{i,j}$ is defined for all $(i,j)\in  ([0,K] \times [0,L]) \cap \Z^2.$

	The conditions on $H$ imply that  for all  integers $0\leq i\leq K$ and $1\leq j\leq L$, there exist integers $M_i$ and $N_j$ so that
	\begin{equation} \label{cond:qp3}
		h_{K,j} = h_{0,j} +  \gamma +\frac{j}{L} + N_j \qquad \text{and} \qquad h_{i,L} = h_{i,0} + M_i.
	\end{equation}	
	In particular,  plugging  $j=0, j=L$ and $i=0, i=K$ into these equations, it is easy to check that
	\begin{equation} \label{eq:winding}
		\begin{aligned}
		M_K-M_0 = 1 + (N_L -  N_0).
		\end{aligned}
	\end{equation}
	We will now show that our choice of branch implies on the one hand, that  $N_L=N_0$, and, on the other hand,  that $M_K=M_0$. This contradicts (\ref{eq:winding}) and would therefore complete the proof.
	
	To show that $N_L=N_0$, we note that due to the relations \eqref{cond:qp3}, we have
	\[
		 N_{j+1} - N_j = \Gamma h_{K,j} - \Gamma h_{0,j} - 1/L,
	\]
	which, by the triangle inequality, together with the conditions $L\geq 2$ and $\abs{\Gamma h_{i,j}} < 1/4$, imply that $\abs{N_{j+1} - N_j} <  1$.
	Since all the $N_j$ are integers, this implies that all $N_j$ are identical and, in particular, that $N_L=N_0$.

On the other hand, to show that $M_K=M_{0}$, let $n_{i,j}\in\Z$ and $|\alpha_{i,j}|<\frac{1}{4}$ be such that
\[
h_{i+1,j}-h_{i,j}=n_{i,j}+\alpha_{i,j}.
\]
Such $n_{i,j}$ and $\alpha_{i,j}$ exist due to our assumption that \eqref{inequality:jump 10} does not hold. Then,
\[
n_{i,j+1}-n_{i,j}=\Gamma h_{i+1,j}-\Gamma h_{i,j} -\alpha_{i,j+1}+\alpha_{i,j}.
\]
As above, the triangle inequality combined with our assumptions  imply that $n_{i,j}=n_{i,j+1}$ for every $i,j$. In particular, we have $n_{i,L}=n_{i,0}$ for every fixed $i \in [0,K] \cap \Z$. Now, since
\[
M_{i+1}-M_i=n_{i,L}+\alpha_{i,L}-n_{i,0}-\alpha_{i,0}=\alpha_{i,L}-\alpha_{i,0},
\]
and since all the $M_i$ are integers, we   conclude that $M_i=M_{i+1}$ for every $i$. This completes the proof.
\end{proof}

\subsection{`Jumps'   of quasi-periodic functions on   subsets of $\Z_d^2$}

In this subsection, we extend Lemma \ref{lemma:jump 1} to show that it  also holds   when the function is restricted to certain subsets of $\Z_d^2$, which we want to treat as if they were sublattices of $\Z_d^2$, even if they, strictly speaking, are not.
To this end, for integers $K,L \in [2,N]$, we  define the functions
\begin{equation} \label{sigmaomega}
		\sigma_s = \Big[ \frac{s N}{K} \Big], \quad s \in [0,K] \cap \Z,\qquad \omega_t = \Big[ \frac{t N}{L} \Big], \quad t \in [0,L] \cap \Z,
\end{equation}
where $[a]$ denotes the integer part of $a$. Note that $\sigma_K=\omega_L=N$.
We can now state the following lemma.

\begin{lemma} \label{douche}
%:
Fix positive integers $N \geq 5$, $K,L\in [2,N]$, and denote $d=N^2$. Let  $W$ be a function defined over $\Z_d^2$ that is  $N$-quasi-periodic up to a constant (see  (\ref{cond:qpuptoconstant})). Denote by $H$ any branch of the argument of $W$, so that $W= \abs{W} \e^{2\pi \im H}$.  Then, there exist $(s,t) \in  ([0,K-1]\times[0,L-1]) \cap \Z^2$  so that either	
	\begin{equation} \label{inequality:jump 1}
		\abs{ \Delta_{(s)} H(\sigma_s,\omega_t)  }  \geq \frac{1}{4} - \frac{1}{N} \quad (\mathrm{mod } \, 1)
	\end{equation}
		or
	\begin{equation} \label{inequality:jump 2}
		\abs{\Gamma_{(t)} H(\sigma_s,\omega_t)  }  \geq \frac{1}{4} - \frac{1}{N} \quad (\mathrm{mod } \, 1).
	\end{equation}
\end{lemma}
\begin{remark}
	Notice that if $K|N$ and $L|N$ then a stronger result follows immediately from  Lemma \ref{lemma:jump 1}. Indeed,  under these conditions, we can apply the lemma directly to the function $\widetilde{H}(s,t) = H(sN/K,tN/L)$ obtaining a jump of size at least $1/4$ instead of $1/4-1/N$.
\end{remark}

\begin{proof}
	Suppose that (\ref{cond:qpuptoconstant}) holds for $W$ with the constant   $\eta =e^{2\pi i \gamma}$. We start by modifying the argument $H(\sigma_{s},  \omega_{t})$ of $W$ to obtain a function that satisfies the conditions of Lemma \ref{lemma:jump 1}.  To this end, for $(s,t)\in\Z^2$, set
$h_{s,t} =    H(\sigma_{s},  \omega_{t})$, and define a function  $\Phi(s,t)$ on $([0,K]\times[0,L]) \cap \Z^2$ as follows:
\begin{equation*}
	  \Phi(s,t) =  \left\{
	\begin{aligned}
		h_{s,t} & \qquad \text{if } (s,t)\in ([0,K-1]\times[0,L]) \cap \Z^2, \\[2mm]
		h_{0,t} +  \gamma + \frac{t}{L}    & \qquad \text{if } s = K.
	\end{aligned} \right.
\end{equation*}
We note that since $\omega_L=N$, and $H$ is the argument of a function that is $N$-periodic in the second variable, we have
\[h_{s,L}=H(\sigma_s,N)=H(\sigma_s,0)=h_{s,0}\qquad (\textrm{mod } 1).\]
Therefore, $\Phi(s,L)=\Phi(s,0)$ (mod $1$) for every $s\in [0,K] \cap \Z$. It follows that $\Phi$ satisfies the conditions of Lemma \ref{lemma:jump 1} on $([0,K]\times[0,L]) \cap \Z^2$,
and, as a consequence, there exists a point $(s,t) \in ([0,K-1]\times[0,L-1]) \cap \Z^2$ so that either
\begin{align*}
	\abs{\Delta \Phi(s,t)} \geq 1/4 \quad \text{or} \quad 	\abs{\Gamma \Phi(s,t)} \geq 1/4 \quad\textrm{(mod 1)}
\end{align*}
hold.
The jumps for $h_{s,t}$ now follow from the jumps of $\Phi$.  Indeed, if the jump is in the vertical direction, or if it is in the horizontal direction at a point $(s,t)$ with $s \leq K-2$, then this is immediate from the definition of $\Phi$. If the jump is in the horizontal direction, at a point $(s,t)$ with $s = K-1$, then we note that,   since $\sigma_K = N$, the $N$-quasi-periodicity  up to a constant of $W$ implies that %\textcolor{blue}{(Do we need to define 'jump'?)}
\begin{align*}
	h_{K,t}
	= H(N,  \omega_t )
	&= H(0 , \omega_t)  +  \gamma +  \frac{\omega_t}{N} \\[2mm]
	&= h_{0,t}  +  \gamma +   \frac{\omega_t}{N} \\[2mm]
	& = \Phi(K,t) + \Big( \frac{\omega_t}{N}  - \frac{t}{L} \Big).
\end{align*}
Since, by definition $\Phi(K-1,t)=h_{K-1,t}$, it follows that
\begin{align*}
|h_{K,t}-h_{K-1,t}|&\geq |\Phi(K,t)- \Phi(K-1,t)|-\Big|\frac{\omega_t}{N}-\frac{t}{L}\Big| \\[2mm]
&\geq \frac{1}{4}-\Big| \frac{\omega_t}{N}-\frac{t}{L}\Big|.
\end{align*}
The lemma now follows from the fact that,
\begin{align*}
	\Abs{\frac{\omega_t}{N}-\frac{t}{L}} &= \Abs{\frac{[tN/L]}{N}-\frac{t}{L} } \\[2mm] &\leq \Abs{\frac{tN/L}{N}-\frac{t}{L} } + \frac{1}{N} = \frac{1}{N}.
\end{align*}
\end{proof}

We  obtain the following corollary of Lemma \ref{douche}.

\begin{corollary}  \label{lem:infidel1}
	Fix an integer $N_0 \geq 5$ and a constant $A>0$. Let $$\delta=2\sqrt{A}\sin\Big( \pi \Big(\frac{1}{4}-\frac{1}{N_0}\Big)\Big).$$ For any integers
	$N \geq N_0$ and $K,L\in [2,N]$, the following holds (with  $d=N^2$).
	If $W$ is an  $N$-quasi periodic function satisfying $A \leq \abs{W}^2$ over the lattice $\Z^2_d$,
	then, for  every  $(u, v) \in \Z_d^2$, there exists at least one point $(s,t)\in ([0,K-1] \times [0,L-1])\cap \Z^2$ such that
	\begin{align}
		&\abs{\Delta_{(s)} W(u+\sigma_{s}  , v+\omega_t)} \geq \delta, \quad \text{or}  \label{baloney10} \\[2mm]
		&\abs{\Gamma_{(t)} W(u+\sigma_s, v+\omega_{t}) } \geq \delta, \label{baloney11}
	\end{align}
where  $\sigma_s, \omega_t$ are   defined in (\ref{sigmaomega}).
	\end{corollary}

\begin{proof}
First, we point out that   we define the argument $\arg(z)$ of a complex number $z$ so that $z=|z|e^{2\pi i\textrm{arg}(z)}$.
Since any translation of a quasi-periodic function is quasi-periodic up to a constant, as defined in (\ref{cond:qpuptoconstant}), it follows from Lemma \ref{douche} that on an $N\times N$  square the argument of $W$ jumps by  more than ${1}/{4} - 1/N$ in at least one of the inequalities (\ref{baloney10}) or (\ref{baloney11}). As the modulus of $W$ is bounded from below by $\sqrt{A}$, the conclusion now follows from basic trigonometry.
\end{proof}

\section{A proof for Theorem \ref{finite blt}} \label{xcom4}

Here we give a proof for Theorem \ref{finite blt} in the general Riesz basis case referred to in Remark \ref{finite blt riesz}.   In the  first subsection below, we reformulate the theorem in terms of the Zak transform. Using this, we proceed to prove the bound from below, and, finally, we prove the bound from above.

\subsection{Measures of  smoothness for finite sequences} \label{secsevc}

Fix $N\in\N$ and set $d=N^2$. For $b \in \ell^d_2$, denote 
\[
\alpha(b,N):=N\sum_{j=0}^{d-1} \abs{\Delta b(j)}^2  +  N\sum_{k=0}^{d-1} \abs{\Delta \mathcal{F}_d{b}(k)}^2,
\]
and
\[
\beta(b,N):=\sum_{m,n=0}^{N-1} \abs{\Delta Z_d(b)\,(m,n)}^2 + \sum_{m,n=0}^{N-1} \abs{\Gamma Z_d (b)\,(m,n)}^2.
\]
Note that with these notations the quantity $\alpha(N)$ defined in the introduction satisfies
\begin{equation*}
	\alpha(N)  = \inf \alpha(b,N)
%\qquad \text{and} \qquad \beta(N)  = \inf \beta(c,N),
\end{equation*}
where the infimum is taken over  $b \in \ell^d_2$  for which   $G(b)$ is an orthonormal basis.%where we use the notations

\begin{proposition} \label{prop:alpha and beta} Let $b\in \ell^d_2$ be such that $|Z_d(b)(m,n)|^2\leq B$ for all $(m,n)\in\Z_d^2$. Then, for all integers $N\geq 2$, we have
	\begin{equation}\label{gaba}
		 \frac{1}{2} \beta(b,N) - 8\pi^2 B   \leq  \alpha(b,N) \leq  2 \beta(b,N) +  8\pi^2 B
	\end{equation}
\end{proposition}
\begin{proof}
We will only prove the right-hand side inequality in (\ref{gaba}) since the left-hand side inequality is proved in the same way.

As the  finite Zak transform commutes with the difference operation $\Delta$, that is
$Z_d(\Delta b)= \Delta Z_d(b)$, and the finite Zak transform is unitary from $\ell_2^d$ to $\ell_2([0,N-1]^2)$, we find that
\[
\|\Delta b\|^2_{\ell_2^d}=\| Z_d(\Delta b)\|^2_{\ell_2([0,N-1]^2)}=\|\Delta Z_d(b)\|^2_{\ell_2([0,N-1]^2)}.
\]
Multiplying the above equation by $d=N^2$, we get
\begin{equation} \label{haba}
N\sum_{j=0}^{d-1} \abs{\Delta b(j)}^2=\sum_{m,n=0}^{N-1} \abs{\Delta Z_d( b)\,(m,n)}^2.
\end{equation}
Similarly,
\begin{equation} \label{laba}
\|\Delta \mathcal{F}_d  b\|^2_{\ell_2^d}=\|\Delta Z_d(\mathcal{F}_d  b)\|^2_{\ell_2([0,N-1]^2)}.
\end{equation}
To relate the expression on the right-hand side to $\Gamma Z_d ( b)$,   we use the relation between the finite Fourier transform and the Zak transform (\ref{Zak and Fourier}) to compute
\begin{align*}
	\abs{\Delta Z_d  (\mathcal{F}_d b) \,(n,-m)} ^2
&= \abs{ \e^{-2\pi \im m(n+1)/d}Z_d ( b)\,(m, n+1) - \e^{-2\pi \im mn/d} Z_d (b)  (m,n)} ^2\\[2mm]
	&\leq  2\abs{ \Gamma Z_d ( b)\, (m,n)}^2 +  2 \abs{(\e^{-2\pi \im m/d}  - 1 )Z_d (b)\, (m,n)}^2 \\[2mm]
	&\leq 2\abs{ \Gamma Z_d ( b)\, (m,n)}^{2} +8\pi^2\frac{m^2}{d^2}B.
\end{align*}
Combining this estimate with (\ref{laba}), and recalling that the Zak transform is $N$-periodic in the second variable, we find that
\begin{align*}
\|\Delta \mathcal{F}_d  b\|^2_{\ell_2^d}&=\frac{1}{N^2}\sum_{m,n=0}^{N-1}|\Delta Z_d(\mathcal{F}_d b)(n,-m)|^2\\[2mm]
&\leq 2\|\Gamma  Z_d( b)\|^2_{\ell_2([0,N-1]^2)}+8\pi^2\frac{1}{N^2}B,
\end{align*}
where, in the last estimate, we used the facts that $m\leq N$ and $d=N^2$.
Multiplying this inequality  by $N^2$, and combining it with (\ref{haba}), the right-hand inequality of  (\ref{gaba})   follows.
\end{proof}
Next, for $A,B>0$, we put
\[
\alpha_{A,B}(N)=\inf\{\alpha(b,N)\}\qquad\textrm{ and }\qquad\beta_{A,B}(N)=\inf\{\beta(b,N)\},
\]
where the infimums are taken over all   $b\in \ell^d_2$  for which the system $G_d(b)$ is a basis with Riesz basis bounds  at least $A$ and at most $B$.

	Proposition \ref{prop:alpha and beta} now implies that with the notations above, the following inequality holds for every  $N \in \N$:
	\begin{equation}\label{alphabeta}
		 \frac{1}{2}\beta_{A,B}(N) - 8\pi^2 B   \leq \alpha_{A,B}(N) \leq  2\beta_{A,B}(N) +  8\pi^2 B.
	\end{equation}

In light of  this inequality, Theorem \ref{finite blt} (as well the version discussed in Remark \ref{finite blt riesz})  can be reformulated as follows.

\begin{theorem}\label{finite blt with betta}
There exist constants $c,C>0$ so that, for all integers $N \geq 2$, we have
\[
c \log N\leq \beta_{A,B}(N)\leq C \log N.
\]
%where ${c,C}$ are constants depending only on $A$ and $B$. %\textcolor{blue}{c and C here as well.}
\end{theorem}

\begin{remark} \label{homeworld} To see that it is necessary to include the  the Riesz basis bounds in the definitions of $\alpha_{A,B}(N)$ and $\beta_{A,B}(N)$ (and that these bounds cannot be replaced by $\ell_2^d$ normalization) consider the following example.  Let $h = \e^{-\pi (x-\tau)^2}$   with $\tau \in (0,1/2) \backslash \Q$. By \cite[Lemma 3.40]{folland1989}, it follows that $Zh$ has exactly one zero on the unit square  located at $(1/2+\tau,1/2)$. Since the first coordinate of this point is irrational, it follows by Proposition \ref{proposition:trump1} that the function $Z_d(b_{N})$ with $b_{N} = S_N P_N h$, where $d=N^2$, does not have a zero on $\Z_d^2$. Consequently, Proposition \ref{gaboriesz} implies that the finite Gabor system $G_d(b_N)$  is a basis for $\ell^d_2$, though the (lower) Riesz basis bounds of these bases decay as $N$ increases.
Straight-forward computations, using only the regularity and decay of the Gaussian, show that there   exist constants $C,D, E>0$,  such that the following hold:
\begin{itemize}
	\item[(i)]  $C \leq\|b_N\|_d\leq D$  for every $N\in\N$.
	\item[(ii)]  $\alpha(b_N,N)\leq E$ as $N \rightarrow \infty$.
\end{itemize}
This example shows that
the constant $c$ in Theorem \ref{finite blt} depends on the lower Riesz basis bounds in the definition of     $\alpha_{A,B}(N)$. Similarly, it may be shown that $C$ depends on the upper Riesz basis bound.
\end{remark}

\subsection{Proof for the lower bound in Theorem \ref{finite blt with betta}} \label{xcom4-1}

\begin{proof}
Given $N\geq 5$, we set $d=N^2$ and let $J\in\N$ be such that $2^J\leq N<2^{J+1}$.  Fix $j\in[0,J-1]  \cap \Z$. Let $\sigma^{(j)}_t$  and $\omega^{(j)}_t$ be as defined in (\ref{sigmaomega}) with $K_j=L_j=2^{J-j}$. That is,
\[
\sigma^{(j)}_t= \omega_t^{(j)} = \Big[t\frac{N}{2^{J-j}}\Big]\qquad t\in \Z.
\]
We note that
\begin{equation}\label{infsup}
2^j \leq \inf\{\Delta\sigma^{(j)}_t\}\leq \sup\{\Delta\sigma^{(j)}_t\} \leq 2^{j+1},
\end{equation}
where  both the infimum and supremum are taken over all $t\in \Z$. Indeed, to see this, consider separately the cases $N=2^J$ and $N>2^J$ and use the fact that all of the numbers involved in these inequalities are integers.

For $u,v \in [0,2^j-1]  \cap \Z$, write
\[
\mathrm{Lat}_j(u)=\big\{u+\sigma^{(j)}_s:  s \in [0,2^{J-j}-1] \cap \Z   \big\}
\]
and
\[
\mathrm{Lat}_j(u,v)=\big\{(u+\sigma^{(j)}_s,v+\sigma^{(j)}_t): s,t \in [0,2^{J-j}-1] \cap \Z \big\}.
\]

Note that due to (\ref{infsup}), if $u_1\neq u_2$ then   $\mathrm{Lat}_j(u_1) \cap \mathrm{Lat}_j(u_2) = \emptyset$, and if $(u_1,v_1)\neq (u_2,v_2)$ then   $\mathrm{Lat}_j(u_1,v_1) \cap \mathrm{Lat}_j(u_2,v_2) = \emptyset$.

Given $b\in \ell_2^d$, we denote $W=Z_d(b)$.  Set $\delta=2\sqrt{A}\sin (\pi/20)$, then, since $N\geq5$ Corollary \ref{lem:infidel1} implies that each of the sets $\mathrm{Lat}_j(u,v)$ contains a point   on which the function $W$ `jumps', i.e., where
% and that a $\Delta_1$ jump for some restriction of the Zak transform translates to a $\Delta_k$ jump in the original Zak transform. Indeed,
%	 \abs{\Delta_1 Z(c_j)(2^k s + u, 2^k t + v)}^2  + \abs{\Gamma_1 Z(c_j)(2^k s + u, 2^k t + v)}^2 \\[2mm]
\begin{equation} \label{inequality:jump}
\begin{aligned}
\delta^2 \leq \abs{\Delta_{(s)}W(u+\sigma^{(j)}_s,v+\sigma^{(j)}_t)}^2 + \abs{\Delta_{(t)}W(u+\sigma^{(j)}_s,v+\sigma^{(j)}_t)}^2.
\end{aligned}
\end{equation}
%where $(m,n) = (2^k s + u, 2^k t + v)$.
Our goal is to collect `jumps' of $W$ that are, in some sense, separated. We do this in an inductive process.

In the first step, let $j=0$.
By Corollary \ref{lem:infidel1}, there exists a point   $(m_0,n_0)$ in $\mathrm{Lat}_0(0,0)$ so that
 \eqref{inequality:jump} holds for this point. Let $\widetilde{S}_0=S_0 = \{(m_0,n_0)\}$.
Next, let $j=1$.
For $u\in\{0,1\}$,    the sets $\mathrm{Lat}_1(u)$ are disjoint, and so at least one of them does not contain the number $m_0$.  Let $u^{1}_1 \in \{0,1\}$ be such that the set  $\mathrm{Lat}_1(u_1^1)$ has this property, and, similarly, let   $v^{1}_1\in\{0,1\}$ be such that $\mathrm{Lat}_1(v^{1}_1)$ does not contain the number $n_0$. By Corollary \ref{lem:infidel1}, there exists a point $(m_1,n_1)$ in $\mathrm{Lat}_{1}(u^{1}_1,v^{1}_1)$ so that
 \eqref{inequality:jump} holds for this point. Let $S_1 = \{(m_1,n_1)\}$, and put $\widetilde{S_1}=S_0\cup S_1$. Note that the two points in $\widetilde{S}_1$ do not have the same value   in either coordinate.

We now consider  the general case. Assume that for some $1 \leq j  \leq J-2$ we have found sets  $S_j$, $\widetilde{S}_j$ and $\widetilde{S}_{j-1}$ so that $\widetilde{S}_j=\widetilde{S}_{j-1}\cup S_j$ and
\begin{itemize}
\item[i.] $\abs{\widetilde{S}_{j-1}}= |S_j|=2^{j-1}$ and  $|\widetilde{S}_j|=2^j$.
\item[ii.] Every point in ${S}_j$ satisfies condition \eqref{inequality:jump}.
\item[iii.] No two   points in $\widetilde{S}_j$ have the same value in either coordinate.
\end{itemize}
 We now construct the sets $S_{j+1}$ and $\widetilde{S}_{j+1}$. Consider the sets $\mathrm{Lat}_{j+1}(u)$ for $u\in [0,2^{j+1}-1]  \cap \Z$. These $2^{j+1}$ sets are disjoint and therefore at least $2^j$ of them do not contain any of the numbers that are   the first coordinates of the points in $\widetilde{S}_j$. We let these sets correspond to  $u^{j+1}_k \in [0,2^{j+1}-1]  \cap \Z$  for $1\leq k\leq 2^j$, and similarly, let $v^{j+1}_k \in[0,2^{j+1}-1]  \cap \Z$, for $1\leq k\leq 2^j$, be   so that no integer in $\mathrm{Lat}_{j+1}(v^{j+1}_k)$ coincide with the second coordinate of any point in   $\widetilde{S}_j$. By Corollary \ref{lem:infidel1}, there exists a point in each of the sets $\mathrm{Lat}_{j+1}(u^{j+1}_k,v^{j+1}_k)$ so that
 \eqref{inequality:jump} holds. Put $S_{j+1}$ to be the set containing all these points and let $\widetilde{S}_{j+1}=\widetilde{S}_j\cup S_{j+1}$. Note that $S_{j+1}$ and $\widetilde{S}_{j+1}$ satisfy all of the conditions (i),(ii) and (iii) above, with  $j$ replaced by $j+1$.

Now, for  a fixed $0 \leq j \leq J-1$, each point $(m,n)\in S_j$ is of the form $(m,n)=(u+\sigma^{(j)}_{s},v+\sigma^{(j)}_{t})$ for some $(u,v)\in [0,2^{j}-1]^2  \cap \Z^2$ and $(s,t)\in [0,2^{J-j}-1]^2  \cap \Z^2$. We observe that condition \eqref{inequality:jump} implies the following  for such a point $(m,n)$ (where we apply the Cauchy-Schwarz inequality and (\ref{infsup}) in the third step):
\begin{align*}
	 \delta^2  &\leq  \abs{\Delta_{(s)} W(u+\sigma^{(j)}_{s},v+\sigma^{(j)}_{t})}^2+ \abs{\Gamma_{(t)}  W(u+\sigma^{(j)}_{s},v+\sigma^{(j)}_{t})}^2 \\[4mm]
	%&\leq  \Bigabs{\sum_{k=0}^{\Delta\sigma^{(j)}_{s}-1}\Delta_{(k)}  W(u+\sigma^{(j)}_{s}+k,v+\sigma^{(j)}_{t})}^2 \\[2mm] & \hspace{50mm}+ \Bigabs{\sum_{\ell=0}^{\Delta\sigma^{(j)}_{t}-1}\Gamma_{(\ell)}  W(u+\sigma^{(j)}_{s},v+\sigma^{(j)}_{t}+\ell)}^2 \\[2mm]
 &=\Bigabs{\sum_{k = u + \sigma_s^{(j)}}^{u + \sigma_{s+1}^{(j)}-1} \Delta  W(k,n)}^2 +   \Bigabs{\sum_{\ell =v+ \sigma_t^{(j)}}^{v+\sigma_{t+1}^{(j)}-1}  \Gamma  W(m,\ell)}^2\Big) \\[2mm]
	  &\leq   2^{j+1}\Big( \sum_{k = u + \sigma_s^{(j)}}^{u + \sigma_{s+1}^{(j)}-1}  \abs{\Delta  W(k,n)}^2 +  \sum_{\ell =v+ \sigma_t^{(j)}}^{v+\sigma_{t+1}^{(j)}-1}  \abs{\Gamma  W(m,\ell)}^2\Big) \\[2mm]
	&\leq 2^{j+1} \Big( \sum_{k=0}^{N-1} \abs{\Delta  W(k,n)}^2 + \sum_{\ell=0}^{N-1} \abs{\Gamma  W(m,\ell)}^2 \Big)
\end{align*}

The last step of the above computation follows by combining the observation that the summands are $N$-periodic, that is, $\abs{\Delta  W(m,n)}  = \abs{\Delta  W(m-N,n)}$ and $\abs{\Gamma  W(m,n)}  = \abs{\Gamma  W(m,n-N)}$, with the fact that,  by \eqref{infsup}, the number of terms in each sum is bounded   by $2^J$, and therefore also by $N$.

%Now, denote $S_0 = \{(m_0^0,n_0^0)\}$, and, for $j\geq 1$,  denote $S_j = \{ (m_k^j,n_k^j) : 1 \leq k \leq 2^{j-1} \}$.
Since no two points of $\widetilde{S}_{J-1}=\cup_{j=0}^{J-1}S_j$ have the same value in either coordinate,  and  $|S_j|=2^{j-1}$, we obtain
\begin{equation}\label{noremorse}
\begin{aligned}
\beta( b,N)=&\sum_{m,n=0}^{N-1}\abs{\Delta  W(m,n)}+\sum_{m,n=0}^{N-1}\abs{\Gamma  W(m,n)} \\[2mm]
\geq&\sum_{(m,n)\in \widetilde{S}_{J-1}}\Big( \sum_{k = 0}^{N-1}  \abs{\Delta  W(k,n)}^2 +  \sum_{\ell = 0}^{N-1}  \abs{\Gamma  W(m, \ell)}^2\Big)\\[2mm]
=&\sum_{j=0}^{J-1}\sum_{(m,n)\in S_{j}}\Big( \sum_{k = 0}^{N-1}  \abs{\Delta  W(k,n)}^2 +  \sum_{\ell = 0}^{N-1}  \abs{\Gamma  W(m,\ell)}^2\Big)\\[2mm]
\geq &\sum_{j=0}^{J-1}\sum_{(m,n)\in S_{j}}\frac{\delta^2}{2^{j+1}}
=\frac{J\delta^2}{4}.
\end{aligned}
\end{equation}
As $J+1 \geq \log N/\log 2$, the desired lower inequality now follows.
\end{proof}
\begin{remark} \label{unremarkable}
Given the restriction $N\geq N_0$, it follows from Corollary \ref{lem:infidel1} that the $\delta$ appearing in the last proof may be choosen to be $\delta = 2\sqrt{A}\sin \pi (1/4 - 1/N_0)$. Moreover,  the inequality $J+1 \geq \log N/\log 2$ yields
	 \begin{align*}
	 	J  %&\geq \log N \cdot \Big( \frac{1}{\log 2} - \frac{1}{\log N} \Big) \\[2mm]
		&\geq  \log N \cdot \Big( \frac{1}{\log 2} - \frac{1}{\log N_0} \Big).
	\end{align*}
	Plugging this into the estimate \eqref{noremorse}, we find that
	\begin{align*}
		\beta(N) %&\geq \frac{J \delta^2}{4} \geq \log N \cdot \frac{\delta^2}{4} \Big( \frac{1}{\log 2} - \frac{1}{\log 5}\Big) \\[2mm]
		&\geq
		\log N \cdot   \Big( \frac{1}{\log 2} - \frac{1}{\log N_0}\Big) \cdot A \sin^2 \pi \Big(\frac{1}{4}- \frac{1}{N_0}\Big).
	\end{align*}
	For $N_0=5$, this  yields the estimate $c \geq A/50$, while as $N_0 \rightarrow \infty$ we get that
\[
c \geq \lim_{N_0\rightarrow \infty}\Big(\inf_{N\geq N_0}\frac{\beta_{A,B}(N)}{\log N}\Big)\geq \frac{A}{2\log 2}.
\]
%\textcolor{blue}{If we change to c and C then we need to change here as well.}
%
\end{remark}

\subsection{Proof for the upper bound in Theorem \ref{finite blt with betta}} \label{xcom4-2}

We consider a function that first appeared  in \cite{benedetto_czaja_gadzinski_powell2003} (see also  \cite{nitzan_olsen2011}).  In these references, this function was used for similar purposes as here, namely, to provide examples of generators of orthonormal Gabor systems with close to optimal localisation.

To define the function, we first  let $\phi :  \R \rightarrow [0,1]$ denote a smooth function so that
\begin{equation*}
	\phi(t) = \left\{ \begin{aligned}  1 & \quad t \geq 1 \\[2mm] 0 & \quad t \leq 0 \end{aligned} \right.,
\end{equation*}
and    $\gamma : (0,1] \rightarrow [0,1]$   a smooth function so that
\begin{equation*}
	\gamma(t) =\left\{ \begin{aligned}  1  & \quad t \in (0, 1/4] \\[2mm] 0 & \quad t \in [1/2,1]
\end{aligned} \right..
\end{equation*}
Using these functions, we define
\begin{equation*}
	H(x,y) = \left\{\begin{aligned}
	\gamma(x) \phi\Big(\frac{y}{x}\Big)  +  \big(1 - \gamma(x) \big) \,  y & \qquad \text{if }x \in (0,1]  \\[2mm]
	1 &\qquad \text{if } x=0
	\end{aligned} \right..
\end{equation*}
Finally, on $[0,1]^2$, we define the function
\begin{equation*}
	G(x,y) =
	 \e^{2\pi \im H(x,y)},
\end{equation*}
which we extend (continuously on $\R^2\setminus\Z^2$)  to a quasi-periodic function on all of $\R^2$ (in a mild abuse of notation, we also denote the quasi-periodic extension   by $G(x,y)$).
\begin{figure}
	\includegraphics[scale=0.9]{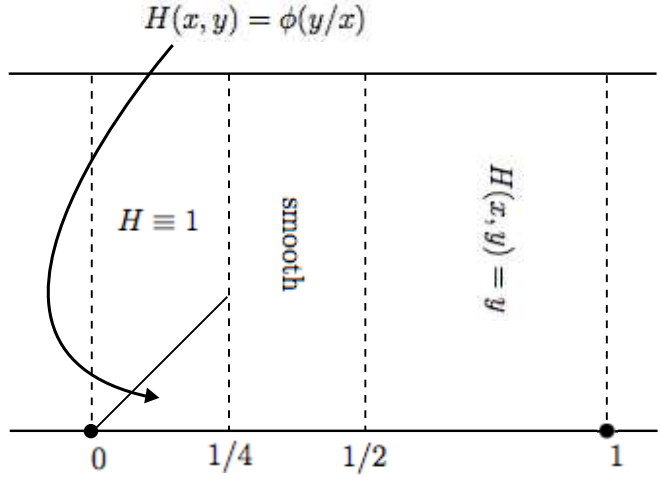}
	\caption{Illustration of the function $H(x,y)$ from \cite{benedetto_czaja_gadzinski_powell2003} used in the proof of the upper bound for Theorem \ref{finite blt with betta}.}
\end{figure}
Since the finite Zak transform is unitary, it follows that
there exists a sequence $b \in \ell^d_2$ of unit norm  so that
\begin{equation} \label{schmabel}
	Z_d(b) (m,n) = G\Big(\frac{m}{N}, \frac{n}{N}\Big).
\end{equation}
In particular, since $G$ is unimodular, it follows by Proposition \ref{gaboriesz} that the Gabor system  $G_d(b)$ is an orthonormal basis for $\ell_2^d$.

The following proposition provides the required estimate from above on $\beta_{1,1}(N)$.
\begin{proposition} \label{upperten}
	Let $b$ be the sequence defined above. Then, there exists  a constant $C>0$ so that for all $N \geq 2$ and $d = N^2$, we have
	\begin{equation} \label{croon}
		\beta(b,N)
\leq  C \log N.
	\end{equation}
\end{proposition}
\begin{proof}
In this proof, we denote by $C$  positive constants which may change from line to line. In light of \eqref{schmabel}, we need to estimate the expression
\begin{equation*}
	\underbrace{\sum_{m,n=0}^{N-1} \Bigabs{\Delta G\Big(\frac{m}{N}, \frac{n}{N}\Big)}^2}_{:=(I)} + \underbrace{\sum_{m,n=0}^{N-1} \Bigabs{\Gamma G\Big(\frac{m}{N}, \frac{n}{N}\Big)}^2}_{:=(II)}.
\end{equation*}

To estimate $(I)$, we make the following partition of  the set $ [0,N-1]^2  \cap \Z^2$:
\begin{align*}
	A_0 &= \{ (m,n) \in \Z^2 :  m \in [0,N/8], n = 0\} \\[2mm]
	A_1 &= \{ (m,n) \in \Z^2 :   m \in [1,N/8] \text{ and } n \in [1,m]   \} \\[2mm]
	A_2 &=  ([0,N-1]^2 \cap \Z^2) \backslash (A_0 \cup A_1)
\end{align*}

On $A_0$, the values of $G(m/N,n/N)$ are constant, so
\begin{align*}
	 \sum_{(m,n) \in A_0} \Bigabs{\Delta G\Big(\frac{m}{N}, \frac{n}{N}\Big)}^2   = 0.% \abs{\Delta G(0,0)}^2= 1.
\end{align*}

On $A_2$, we use the fact that  the function $G$ is   $C^\infty$ on  $\R^2 \backslash \Z^2$, and moreover, that on the set $\{(x,y)\in [0,1]^2:y\geq x\vee x\geq 1/8\}$ both  $G$, and its derivatives, are continuous.
Indeed, this means that we are justified in using the Mean Value theorem for $(m,n) \in A_2$ to make the estimate
\begin{align*}
	\Bigabs{\Delta G\Big(\frac{m}{N}, \frac{n}{N}\Big)} &= \frac{1}{N} \cdot \Bigabs{(\partial_x G)\Big(\mu_{m,n}, \frac{n}{N}\Big)} \leq \frac{C}{N},
\end{align*}
where $\mu_{m,n} \in (m/N,(m+1)/N)$ and $C = C(G) > 0$ is a constant  not depending on $N$. It follows immediately that
\begin{equation*}
	 \sum_{(m,n) \in A_2} \Bigabs{\Delta G\Big(\frac{m}{N}, \frac{n}{N}\Big)}^2
	\leq
	\frac{1}{N^2} \cdot C \abs{A_2 } \leq C.
\end{equation*}

On  $A_1$, we do the computation
\begin{equation*}
	 \sum_{(m,n) \in A_1}  \Bigabs{\Delta G\Big(\frac{m}{N}, \frac{n}{N}\Big)}^2
	=
	\underbrace{  \sum_{m=1}^{ [N/8]} \, \, \sum_{n=1}^{m}   \Bigabs{(\partial_x G\Big(\mu_{m,n}, \frac{n}{N}\Big)}^2  \frac{1}{N^2} }_{:=\,(*)},
\end{equation*}
 where  $\mu_{m,n} \in (m/N,(m+1)/N)$ as before. To estimate $(*)$, we compute
%\begin{align*}
%	(\partial_x Zg)(x, n/N) = \partial_x (\e^{2\pi \im H(x,n/N)}) &= 2\pi (\partial_x H)(x,n/N) \e^{2\pi \im H(x,n/N)},
%\end{align*}
%where
\begin{align*}
	\Bigabs{(\partial_x G)\Big(x, \frac{n}{N}\Big)} &= 2\pi \Bigabs{(\partial_x H)\Big(x, \frac{n}{N}\Big)} \\[2mm]
	&= 2 \pi \Bigabs{\partial_x \phi\Big(\frac{n}{Nx}\Big)} \\[2mm]
	&= 2 \pi \Bigabs{  \phi'\Big(\frac{n}{Nx}\Big)} \cdot \frac{n}{Nx^2}.
\end{align*}
This allows the bound
\begin{align*}
	(\ast)
%	&=
%	  \sum_{m =1}^{ [N/8]} \, \, \sum_{n = 1 }^{m}   \abs{(\partial_x G(\mu_{m,n}, n/N)}^2  \frac{1}{N^2} \\[2mm]
	&=   4 \pi^2 \sum_{m=1  }^{[N/8]}   \sum_{n = 1 }^{m}  \Bigabs{  \phi'\Big(\frac{n}{N \mu_{m,n}}\Big)}^2 \frac{n^2}{N^4 \mu_{m,n}^4} \\[2mm]
	&\leq \frac{4 \pi^2\norm{\phi'}_\infty^2}{N^4}   \sum_{m=1  }^{[N/8]}  \sum_{n = 1 }^{m}  \frac{ n^2}{(m/N)^4}  \\[2mm]	
	&\leq  4 \pi^2 \norm{\phi'}_\infty^2  \sum_{m=1  }^{[N/8]} \frac{1}{m^4} \cdot (m+1)^3    \\[2mm]
%	&\leq  {4 \pi^2C\norm{\phi'}_\infty^2}   \sum_{m=1  }^{[N/8]} \frac{1}{m}  \\[2mm]
	&\leq  4 \pi^2\norm{\phi'}_\infty^2   \Big(  C + \log  N \Big).
\end{align*}

To estimate (II), we consider the corresponding sums over $A_2$ and over $B = A_0\cup A_1$. We skip the estimates for $A_2$, which are completely analogous to the corresponding estimates made above. Instead, we turn our  focus to the estimate for $B$.
As above, we begin by  using the Mean Value Theorem to write
\begin{equation*}
	\sum_{(m,n) \in B}  \Bigabs{\Gamma G\Big(\frac{m}{N}, \frac{n}{N}\Big)}^2
	=
	\underbrace{  \sum_{m=1}^{[N/8]} \, \, \sum_{n = 0}^{m}   \abs{(\partial_y G)\Big(\frac{m}{N}, \nu_{m,n}\Big)}^2 \frac{1}{N^2} }_{:=\,(**)},
\end{equation*}
where $\nu_{m,n}\in [n/N,(n+1)/N]$. To estimate $(**)$, we compute
\begin{align*}
	\Bigabs{(\partial_y G)\Big(\frac{m}{N}, y\Big)} &= 2\pi \Bigabs{(\partial_y H)\Big(\frac{m}{N}, y\Big)} \\[2mm]
	&= 2\pi \Bigabs{\partial_y \phi\Big(\frac{Ny}{m}\Big)} \\[2mm]
	&= 2\pi \Bigabs{ \phi'\Big(\frac{Ny}{m}\Big)} \cdot \frac{N}{m}.
\end{align*}
This allows the estimate
%\textcolor{blue}{here I think that the third equality (with $m+1$ in the nominator) can be removed, but again, you decide}
\begin{align*}
	(**) &=  4\pi^2 \sum_{m=1}^{[N/8]} \, \, \sum_{n = 0}^{m}  \Bigabs{ \phi'\Big(\frac{N \nu_{m,n}}{m}\Big)}^2 \cdot \frac{1}{N^2} \cdot  \frac{N^2}{m^2} \\[2mm]
	&\leq  4\pi^2 \norm{\phi'}^2_\infty \sum_{m=1}^{[N/8]} \, \, \sum_{n = 0}^{m}  \frac{1}{m^2} \\[2mm]
%	&=  {4\pi^2 \norm{\phi'}^2_\infty} \sum_{m=1}^{[N/8]}    \frac{m+1}{m^2}   \\[2mm]
	&\leq 4\pi^2 \norm{\phi'}^2_\infty  \sum_{m=1}^{[N/8]} \frac{ m+1}{ m^2} \\[2mm]
	&\leq  4\pi^2 \norm{\phi'}^2_\infty  \Big( C + \log  N \Big).
\end{align*}
\end{proof}
\begin{remark} \label{unremarkabletoo}
	To determine a bound for the constant $C$ of inequality \eqref{croon}, observe that we can actually choose $\phi$ to be piecewise linear, and therefore to satisfy  $\norm{\phi'}_\infty^2 \leq 1$.
Moreover, observe that
	\begin{equation*}
		(*) + (**) \leq 8\pi^2 (C_0 + \log N ),
	\end{equation*}
	where $C_0$ is some positive constant. Since the remaining sums over $A_2$ only contribute to  $C_0$, we conclude that asymptotically,

\[
C \leq  \lim_{N_0\rightarrow \infty} \Big(\sup_{N\geq N_0}\frac{\beta_{1,1}(N)}{\log N}\Big) \leq 8 \pi^2.
\]
\end{remark}

\section{A quantitative Balian-Low type theorem in finite dimensions} \label{xcom5}

In this section, we prove  a finite  dimensional version of the quantitative Balian-Low inequality. For the most part,  we follow the main ideas appearing in our paper \cite{nitzan_olsen2013}.

\subsection{Auxiliary results} The estimate in the following lemma is not optimal, but rather,  chosen to simplify the presentation. %  (it is also the reason we make the Remark \ref{hasta la vista}).
\begin{lemma}\label{rho-lemma}
	Let   $\rho : \R \rightarrow \R$ be the inverse Fourier transform of
	\begin{equation*}
		 \hat{\rho}(\xi) = \left\{ \begin{aligned}  1 & \qquad \abs{\xi} \leq 1/2,  \\[2mm]  2(1 - \xi\operatorname{sgn}(\xi)) & \qquad 1/2 \leq \abs{\xi} \leq 1, \\[2mm] 0 & \qquad \abs{\xi} \geq 1.     \end{aligned} \right.
	\end{equation*}
	Then
	\begin{equation*}
		\int_\R \abs{\rho'(t)} \dif t \leq 10.
	\end{equation*}
\end{lemma}

\begin{proof}
Fix $b>0$ to be chosen later. We make the split
\begin{equation*}
	\int_\R \abs{\rho'(t)} \dif t
	=
	\underbrace{\int_{-b}^b \abs{\rho'(t)} \dif t}_{=(I)} + \underbrace{\int_{\abs{t} \geq b} \abs{\rho'(t)} \dif t}_{=(II)}.
\end{equation*}
To estimate (I), we compute
\[
\begin{aligned}
|\rho'(t)| &= |\mathcal{F}^{-1} \Big(2\pi \im \xi \cdot \hat{\rho} \Big)(t)| \\[2mm]
&\leq 2\pi\|\xi \cdot \hat{\rho}\|_{L^1(\R)}
%&\leq 4\pi\Big(\int_0^{1/2}\xi \dif \xi+ \int_{1/2}^1\xi(2-2\xi) \dif \xi\Big)\\
= {2\pi} \cdot \frac{7}{12} =  \frac{7\pi}{6} ,
\end{aligned}
\]
whence
\[
(I)\leq \frac{7\pi}{6} \cdot 2b= \frac{7\pi b}{3}.
\]

We turn to estimate (II). Integrating by parts, %, and using the fact that $\rho''=0$ almost everywhere,
we find that
\begin{align*}
	\rho'(t) = \mathcal{F}^{-1} \Big(2\pi \im \xi \cdot \hat{\rho} \Big)(t)
	&=   2 \pi \im \int_\R \xi   \hat{\rho}(\xi) \cdot \e^{2\pi \im \xi t} \, \dif \xi \\[2mm]
	&=   - 2 \pi \im \int_\R  (   \hat{\rho}+  \xi\hat{\rho}')  \cdot \frac{\e^{2\pi \im \xi t}}{2\pi \im t} \, \dif \xi \\[2mm]
	%&=    2 \pi \im \int_\R  (2\hat{\rho}'+ {\xi} \hat{\rho}'')  \cdot \frac{\e^{2\pi \im \xi t}}{(2\pi \im t)^2} \,  \dif \xi \\[2mm]
%	& = \frac{1}{\pi \im t^2} \int_{\R}  \hat{\rho}' \e^{2 \pi \im t \xi}  \, \dif \xi.
	&=   2\pi \im \int_\R  (   2\hat{\rho}' +  \xi\hat{\rho}'')    \frac{ \e^{2\pi \im \xi t}}{(2\pi \im t)^2}  \, \dif \xi \\[2mm]
	&= \frac{1}{\pi \im t^2}  \int_\R     \hat{\rho}'(\xi) \e^{2\pi \im \xi t} \dif \xi + \frac{1}{\pi   t^2} \Big( 2 \sin (2\pi t) -  \sin(\pi t)\Big).
\end{align*}
%where the sine-terms are due to the Dirac delta functions appearing at $\xi  =\pm 1/2$ and $\xi = \pm 1$  with mass $\mp 2$, respectively, in the distribution   $\hat\rho''$.
This yields the bound
\begin{equation} \label{penis}
	\abs{\rho'(t)} \leq \frac{1}{\pi t^2}  \Big( \int_{\R}   \abs{\hat{\rho}'} \dif \xi  +2 + 1 \Big) =\frac{ 5}{\pi t^2},
\end{equation}
and the estimate
\[
(II)\leq \frac{10}{\pi}\int_b^{\infty}\frac{dt}{t^2}=\frac{10}{\pi b}.
\]
We compute the value of $b$ which minimizes the bound we obtained for (I)+(II) and find that this value is  $b=  \sqrt{30/7}/\pi$, which gives us the estimate $(I)+(II)\leq  9.67<10$.

\end{proof}
%
%{Next, we have the following.} \textcolor{blue}{Note that below I changed $N_0$ to 100, this is partially due to the fact that I allowed the rho prime estimate to go as high as 6.15 and partially due to the fact that, at least by my calculator, your 20 was too optimistic. I don't care about this $N_0$ much at this point, but if you do then I guess we need more effort. Also, I removed the $N_0$ notation from the formulation, I think, as you say, simpler is better. I removed also the max, as I thought it did not contribute a lot compared to how difficult this lemma is to read. I indicate in a few places in the text that our constants here are not optimal, I think that if someone will really want to get better constants they will notice that they can do the 'max' trick. Is that OK? Finally, I gave a different structure to the formulation of the lemma (with 'items') in the hope that this will make it seem more accessible. If you have a better way of presentation to suggest I will be even happier :)}.

\begin{lemma}\label{conv-lemma}
Let $A, B>0$ and $N\geq    200 \sqrt{B/A}$. There exist positive constants $\delta=\delta(A)$ and $C=C(A,B)$ such that   the following holds (with $d=N^2$).  Let
\begin{itemize}
\item[(i)]   \quad $Q, R \in \Z$ such that $1\leq Q,R    \leq (N/16) \cdot  \sqrt{A/B}$,
\item[(ii)]   \quad $\phi,\psi \in \ell_2^d$ such that $\sum_n|\Delta\phi(n)|\leq  10   R$ and $\sum_n|\Delta\psi(n)|\leq  10 Q$,
\item[(iii)] \quad $b\in \ell_2^d$ such that $A \leq |Z_d (b)|^2 \leq B$.
\end{itemize}
Then, there exists a set $S\subset ([0,N-1]\cap\Z)^2$ of size $|S|\geq C N^2/ QR$
 such that all $(u,v)\in S$ satisfy either
\begin{align}\label{conv-ineq-1}
|Z_d(b)(u,v)-Z_d(b\ast \phi)(u,v)|\geq\delta, \qquad \text{or} \\[2mm]
\label{conv-ineq-2}
|Z_d( \mathcal{F}_d  b)(u,v)-Z_d(( \mathcal{F}_d b)\ast \psi)(u,v)|\geq\delta.
\end{align}
\end{lemma}

\begin{proof}[Proof of Lemma \ref{conv-lemma}]
For the given $A$ and for $N_0=200$, let $\delta_1 = 2\sqrt{A}\sin (\pi/4-\pi/{200})$ be the constant from Corollary \ref{lem:infidel1}. Notice that the integers $Q, R$ and $N$ satisfy
\[
\frac{\sqrt{B}}{\delta_1}\max\Big\{ {\frac{200 R}{9}}, {\frac{200 Q}{9}}, {80\pi} \Big\}  {\leq} N.
\]
Therefore, there exist integers $K$ and $L$ that satisfy
\begin{equation} \label{lynch}
	{\frac{200\sqrt{B} R}{9\delta_1}} {\leq} K \leq N  \qquad \text{and} \qquad \frac{\sqrt{B}}{\delta_1}  \max\Big\{ {\frac{200  Q}{9}, 80\pi} \Big\} {\leq} L \leq N.
\end{equation}
We choose $K, L$ to be the smallest such integers.
%\textcolor{blue}{After fixing all constants, we should make sure that indeed all of the needed inequalities hold.}

For $s,t \in \Z$,   let $\sigma_s$ and $\omega_t$ be as defined in (\ref{sigmaomega}), that is,
%{(not absolutely sure if I am correct, but I replaced $K$ by $K-1$ and $L$ by $L-1$ here and in some places below... also $N$ by $N-1$...)}
\[
		\sigma_s = \Big[ \frac{s N}{K} \Big],\qquad \textrm{and}\qquad\omega_t = \Big[ \frac{t N}{L} \Big].
\]
Note that $\sigma_K=\omega_L=N$ and denote $\Sigma = \inf\{\Delta_{(s)}\sigma_s\}$ and $\Omega= \inf\{\Delta_{(t)}\omega_t\}$.
%Using the same argument as was used to establish \eqref{infsup},  we n
The conditions on $K$ and $L$ imply that, for some constant $C_1=C_1(A,B)$,
\begin{equation*}
	\Sigma\cdot\Omega \geq C_1 \frac{ N^2}{QR}.
\end{equation*}
%{( I am too tired to think about it now, bu do we have a formula for $\eta_1$?)}\textcolor{blue}{(Something like $\delta_1/8b$?)}

For $(u,v)\in ([0,\Sigma -1]\cap\Z)\times ([0,\Omega -1]\cap\Z)$, write
\[
\mathrm{Lat}(u,v)=\big\{(u+\sigma_s,v+\omega_t): (s,t)\in ([0,K-1]\cap\Z)\times([0,L-1]\cap\Z)\big\}.
\]
and
\[
\mathrm{Lat^*}(u,v)=\big\{(N-v-\omega_t,u+\sigma_s): (s,t)\in ([0,K-1]\cap\Z)\times([0,L-1]\cap\Z)\big\}.
\]
Note that if $(u_1,v_1)\neq (u_2,v_2)$ then it follows from the definition of $\Sigma$ and $\Omega$ that  $\mathrm{Lat}(u_1,v_1) \cap \mathrm{Lat}(u_2,v_2) = \emptyset$ and $\mathrm{Lat^*}(u_1,v_1) \cap \mathrm{Lat^*}(u_2,v_2) = \emptyset$. We will show that for each $(u,v)$, either the set $\mathrm{Lat}(u,v)$ or the set  $\mathrm{Lat^*}(u,v)$ contains a point from $([0,N-1]\cap\Z)^2$ which satisfies condition (\ref{conv-ineq-1}) or (\ref{conv-ineq-2}), respectively. Putting $C=C_1/2$ our proof will then be complete.

So,  fix $(u,v)\in ([0,\Sigma -1]\cap\Z)\times ([0,\Omega -1]\cap\Z)$. Due to Corollary \ref{lem:infidel1}, there exists at least one point $(s,t)\in ([0,K-1]\cap\Z)\times([0,L-1]\cap\Z)$ such that
(\ref{baloney10}) or (\ref{baloney11}) hold with $W= Z_d(b)$.

Assume first  that (\ref{baloney10}) holds. Lemma \ref{convolution inequality 0},  the estimate $|\Delta_{(s)}\sigma_s|\leq 2N/K$, and condition (\ref{lynch}), imply that
%\textcolor{blue}{In one way or another we should remind now who is $Q$. I was not sure whether to copy the condition from the formulation of the lemma or to refer to it in some way?}
\[
		\abs{\Delta_{(s)} Z_d (b \ast \phi)(u+\sigma_s,v+\omega_t)}
		%&\leq \frac{\norm{Z_\Lambda c\,}_{L^\infty}}{\sqrt{ab}} \sum_{j=0}^{ab-1} \abs{\Delta_{D} \phi_{j}} \\[2mm]
		\leq  \frac{2N/K}{N}\cdot\sqrt{B}\cdot {10} R \leq \frac{9\delta_1}{10}.
	\]
It follows that either $(u+\sigma_s,v+\omega_t)$ or $(u+\sigma_{(s+1)},v+\omega_t)$ satisfy (\ref{conv-ineq-1}) with $\delta=\delta_1/20$. It may happen that this chosen point  does not belong to $([0,N-1]\cap\Z)^2$, that is, it is of the form $(u+N,v+\omega_t)$. In that case, due to the $N$-quasi-periodicity of the Zak transform, the point $(u,v+\omega_t)$ satisfies the same inequality, and is in $([0,N-1]\cap\Z)^2$.

Assume now that (\ref{baloney11}) holds for $(u,v)$ and $(s,t)$. Then (\ref{Zak and Fourier}), the $N$-quasi-periodicity of $Z_d(b)$, and the estimate $|\Delta_{(t)}\omega_t|\leq  {2N/L}$, imply that
\begin{equation*}
\begin{aligned}
&|\Delta_{(t)} Z_d(\mathcal{F}_d b)(N-v-\omega_t,u+\sigma_s)| \\[2mm]
&=\Bigabs{Z_d(b)(u+\sigma_s,v+\omega_{t+1})- e^{2\pi i\frac{(\omega_{t+1}-\omega_t)(u+\sigma_s)}{d}}Z_d(b)(u+\sigma_s,v+\omega_{t})}\\[2mm]
&\geq |\Gamma_{(t)} Z_d(b)(u+\sigma_s,v+\omega_{t})|- \sqrt{B}|e^{2\pi i\frac{(\omega_{t+1}-\omega_t)(u+\sigma_s)}{d}}-1|
\\[2mm]
&\geq \delta_1- {\frac{4\pi \sqrt{B}}{L}} \geq \frac{ 19\delta_1}{20}.
\end{aligned}
\end{equation*}
On the other hand, as in the previous case, we have
\[
		\abs{ \Delta_{(t)} Z_d(\mathcal{F}_db\ast \psi)(N-v-\omega_t,u+\sigma_s)}
		%&\leq \frac{\norm{Z_\Lambda c\,}_{L^\infty}}{\sqrt{ab}} \sum_{j=0}^{ab-1} \abs{\Delta_{D} \phi_{j}} \\[2mm]
		\leq \frac{{20}\sqrt{B} Q}{L} < \frac{9\delta_1}{10}.
	\]
It follows that either $(N-v-\omega_t,u+\sigma_s)$ or $(N-v-\omega_{(t+1)},u+\sigma_s)$ satisfy (\ref{conv-ineq-2}) with $\delta=\delta_1/40$. It may happen that this chosen point is the point $(-v,u+\sigma_s)$. In that case, due to the $N$-quasi-periodicity of the Zak transform, the point $(N-v,u+\sigma_s)$ satisfies the same inequality, and is in $([0,N-1]\cap\Z)^2$.
\end{proof}

\subsection{A proof for Theorem \ref{finite quantitative BLT 1}}

We are now ready to prove Theorem \ref{finite quantitative BLT 1}.  In fact, we prove the  more general version referred to in Remark \ref{finite quantitative BLT 2}  which we formulate as follows:

\begin{theorem} \label{finite quantitative BLT 1 -rb}
	Let $A,B>0$. There exists a constant $C=C(A,B)>0$ so that the following holds. Let $N \geq 200 \sqrt{B/A}$ and let $b \in \ell_2^d$ (where $d=N^2$) be  such that $G_d(b)$ is a  basis in $\ell_2^d$ with Riesz basis bounds $A$ and $B$. Then, for all  positive integers $1 \leq Q, R \leq  ( N/{16})\cdot \sqrt{A/B}$,  we have	
	\begin{equation*}
		\frac{1}{N}  \sum_{j=N  Q}^{d-1} \abs{b(j)}^2 + \frac{1}{N} \sum_{k= N  R}^{d-1} \abs{\mathcal{F}_d b(k)}^2   \geq \frac{C}{ QR}.
	\end{equation*}
\end{theorem}

\begin{proof}[Proof of Theorem \ref{finite quantitative BLT 1}]
Let $\rho: \R \rightarrow {\R}$ be  as in Lemma \ref{rho-lemma} and put $\Phi(t) =  R \rho(Rt)$ and $\Psi(t) = Q \rho( Qt)$. Denote $\phi = S_N P_N \Phi$ and $\psi = S_N P_N \Psi$.   By Lemma \ref{derder}, in combination with Lemma \ref{rho-lemma}, it follows that
$\sum_n|\Delta\phi(n)|\leq {10}  R$ and $\sum_n|\Delta\psi(n)|\leq  {10} Q$.
As a consequence, the integers $Q,R$ and $N$, as well as the functions $\phi,\psi$ and $b$, all satisfy the requirements of Lemma \ref{conv-lemma}.

As $Q,R < N/2$, by Proposition \ref{proposition:trump1}, and Remark \ref{hasta la vista}(i), we have $0\leq \mathcal{F}_d\phi, \mathcal{F}_d\psi \leq 1$ over $\ell_2^d$. Moreover, we also have $\mathcal{F}_d\phi(j)=1$ for $j\in  [0, RN/2] \cup [N^2- RN/2+1,N^2]$ and $\mathcal{F}_d\psi(j)=1$ for $j\in [0, QN/2] \cup [N^2- QN/2+1,N^2]$.  
It therefore follows from Lemma \ref{conv-lemma}, and the fact that  both the finite Zak transform and the finite Fourier transform are unitary, that, for some constant $C>0$, 
\[
\begin{aligned}
\frac{C}{QR}&\leq \|Z_d(b)-Z_d(b\ast \phi)\|_{\ell_2([0,N-1]^2)}^2 \\[2mm]
& \hspace{30mm}+\|Z_d( \mathcal{F}_d b)-Z_d(  \mathcal{F}_d b\ast \psi)\|_{\ell_2([0,N-1]^2)}^2\\[4mm]
&=\|b-b\ast \phi\|_{\ell^2_d}^2+\|  \mathcal{F}_d b-\mathcal{F}_d b \ast \psi\|_{\ell^2_d}^2\\[2mm]
&= \|\mathcal{F}_d b (1 - \mathcal{F}_d \phi) \|_{\ell^2_d}^2+\|   b (1-   \mathcal{F}_d\psi )\|_{\ell^2_d}^2\\[2mm]
&\leq  \frac{1}{N} \sum_{k= [NR/2]+1}^{d-1-[NR/2]} \abs{\mathcal{F}_d b(k)}^2 + \frac{1}{N}  \sum_{j=[N Q/2]+1}^{d-1-[NQ/2]} \abs{b(j)}^2.
\end{aligned}
\]
The result now follows by applying a suitable   time-frequency translate to the sequence $b$.
\end{proof}
\begin{comment}
\begin{remark} \label{vrooom}
	By making a careful analysis of the above proof, it is possible to improve the constants in the condition of the above lemma. Indeed,  if we replace the lower bounds $\delta = \delta_1/20$ by $\delta =  (\epsilon/2) \delta$ for some small $\epsilon >0$, then we may replace   condition \eqref{lynch} by
	\begin{equation*}
		 \frac{2bQ}{(1-\epsilon)\delta_1} \leq K \leq N  \qquad \text{and} \qquad  \max\Big\{  \frac{2bR}{(1-\epsilon)\delta_1}, \frac{4\pi b}{(1-\epsilon/2)\delta_1} \Big\} \leq L \leq N.
	\end{equation*}
	Note that, formally, if we let $N_0\rightarrow \infty$ and  $\epsilon \rightarrow 0$, then the condition on $Q,R,N$ would read
	$N \geq \sqrt{2} \mu \frac{b}{a}\max \{Q,R\}$. By choosing $N_0$ and $\epsilon$ appropriately, we can get as close to this condition as we desire.
\end{remark}
\end{comment}

\section{Applications to the continuous setting} \label{seksjon 6} \label{xcom6}
In this section we show that both the  classical and quantitative Balian-Low theorems  follow from their finite dimensional analogs.

\subsection{Relating continuous and finite signals -- revisited} \label{revisited}
We start by extending Proposition \ref{proposition:trump1} to the space $L^2(\R)$. We do this in four steps. In the first, we introduce   some additional notations. To this end, fix $N\in\N$, $N\geq 2$, and let $d=N^2$.

\textbf{ Step I:} By $(L^2[0,1/N]^2)^d$,  we denote
the space of all  d-tuples  $\{   \phi(j)\}_{j=0}^{d-1}$ with function entries  $\phi(j) \in L^2([0, {1}/{N}]^2)$, equipped with the norm  given by 
\[
\|\{  \phi(j)\}   \|_{(L^2[0,1/N]^2)^d}^2=N\sum_{j=0}^{d-1}\| \phi(j)\|^2_{L^2([0,1/N]^2)}.
\]
Note that the factor $N$  appears in the norm in order to take    the measure of $[0, {1}/{N}]$ into account. Similarly,  by $(L^2[0,1/N]^2)^{N\times N}$,    we denote
the space of all $N\times N$  matrices $\{  \phi(j,k)\}_{j,k=0}^{N-1}$ with function entries  $\phi(j,k) \in L^2([0, {1}/{N}]^2)$, equipped with  norm  given by
\[
\|\{  \phi(j,k) \}\|_{(L^2[0,1/N]^2)^{N\times N}}^2=\sum_{j,k=0}^{N-1}\|  \phi(j,k) \|^2_{L^2([0,1/N]^2)}.
\]
%We understand the notation $T \{h(j)\}$, {where} $T$ {is an operator} and $\{h(j)=h(j)(u,v)\}\in (L^2[0,{1}/{N}]^2)^d$ {a d--tuple}, to mean that $T$ operates  with respect to  the  variable $j$ with $(u,v)$ being considered fixed.

\textbf{  Step II:}  We consider functions $h(u,v;t)$ defined over $[0, {1}/{N}]^2\times \R$,
that are $N$-periodic with respect to  the variable $t$,  and are such that, for every fixed $t_0$, the restriction $h(u,v;t_0)$ is well defined almost everywhere and belongs to $L^2([0,1/N]^2)$.
Observe that the operator
\begin{equation*}
	S_N \, h  := \{ h(u,v;j/N) \}_{j=0}^{d-1}
\end{equation*}
trivially satisfies
\begin{align}
	&S_N h \in (L^2[0,{1}/{N}]^2)^d , \label{trivial1} \\[2mm]
	 &\mathcal{F}_dS_N h \in (L^2[0,{1}/{N}]^2)^d, \label{trivial2}\\[2mm]
	 &Z_d S_N h  \in (L^2[0, {1}/{N}]^2)^{N\times N}.\label{trivial3}
\end{align}
Above, we understand the notations $\mathcal{F}_d S_N h$ and $Z_d S_N h$ to mean that $\mathcal{F}_d$ and $Z_d$ operate on  $S_Nh$ with respect to the variable $j$ with $(u,v)$ being considered fixed.

\textbf{ Step III:}  Let $f\in L^2(\R)$. For $(u,v)\in [0, {1}/{N}]^2$, we define the function
$$f_{(u,v)} (t):= \e^{2\pi \im  vt} f(t+u),\qquad  t\in\R$$
and  formally  put
\begin{equation*}
	P_N {f}_{(u,v)}(t)  := \sum_{\ell= - \infty}^\infty f_{(u,v)}(t+ \ell N).
\end{equation*}
Note that if  $f$ is in the Schwarz class $\mathcal{S}(\R)$, then the function $h(u,v;t):=P_N {f}_{(u,v)}(t)$ satisfies the conditions on $h(u,v;t)$ described in Step II. The following lemma shows that this is true for all $f\in L^2(\R)$.
\begin{comment}
Moreover,  $S_N P_N f_{(u,v)}$ is a vector valued function of the form $g(u,v;k)$. It is an immediate consequence of the following lemma that, for $f \in L^2(\R)$, the latter function is in $(L^2[0,1/N]^2)^{d}$.
\end{comment}
%
\begin{lemma}  For $f\in L^2(\R)$, let $f_{(u,v)} (t)$ be the function  defined above. Then, for  every fixed $t_0 \in \R$, we have
\begin{equation*}
	P_N {f}_{(u,v)}(t_0)   \in L^2([0,1/N]^2).
\end{equation*}
That is, the series defining $P_N f_{(u,v)}(t_0)$ converges in the norm of $L^2([0,1/N]^2)$.\end{lemma}
\begin{proof}
The lemma follows from the following computation,  where, in the first step, we use the fact that $\{\sqrt{N}\e^{2\pi \im  N\ell v}\}_{\ell \in \Z}$ is an orthonormal basis over $[0,1/N]$:
\begin{align*}
  \int_0^{\frac{1}{N}}\int_0^{\frac{1}{N}} \Bigabs{\sum_{|\ell|>L} f(t_0+u +\ell  N)&\e^{2\pi \im  N\ell v}}^2  \dif v \dif u \\[2mm]
&=
\frac{1}{N}\sum_{|\ell|>L}\int_0^{\frac{1}{N}}| f(t_0+u +\ell N)|^2 \dif u\\[2mm]
&\leq \frac{1}{N}\int_{|x-t_0|\geq L/2}|f(x)|^2\dif x% \\[3mm]
\rightarrow 0, \quad \text{as } L \rightarrow \infty. %\\[-3mm]
\end{align*}
\end{proof}

It follows that, for $f\in L^2(\R)$, the operator $S_NP_Nf_{(u,v)}$ is well defined and that conditions \eqref{trivial1}, \eqref{trivial2}, \eqref{trivial3} hold with $h (u,v;t)=P_Nf_{(u,v)}(t)$.

\textbf{  Step IV:}  For   $f$ in $L^2(\R)$, we understand the notations $\mathcal{F} f_{(u,v)}(t)$ and $Z f_{(u,v)}(t)$ to mean that the Fourier transform and the Zak transform are taken with respect to the variable $t$, with $(u,v)$ being fixed. We now give our  extension of Proposition \ref{proposition:trump1}.
\begin{proposition}\label{proposition:trump2}
	\hspace{2mm}
Let $f\in L^2(\R)$. Then, the following hold.
\begin{itemize}
\item[(i)]   For all $N\in\N$ and $(m,n)\in\Z^2$, we have
		\begin{equation} \label{potus2}
			Z_{d} (S_N {P_N f_{(u,v)}})(m,n) = Zf_{(u,v)}(m/N, n/N),
		\end{equation}
where the equality holds in  the sense of ${L^2([0,{1}/{N}]^2)}$.
		\item[(ii)] For all $N\in\N$, we have
		\begin{equation} \label{eq:*2}
			  \mathcal{F}_{d}S_N P_N {f_{(u,v)}} = S_N P_N \mathcal{F} f_{(u,v)},
		\end{equation}
		where the equality holds in the sense of  $(L^2[0,{1}/{N}]^2)^d$.
	\end{itemize}
\end{proposition}
\begin{proof}
First, Proposition \ref{proposition:trump1} implies that \eqref{potus2} and  \eqref{eq:*2} hold for $f \in \mathcal{S}(\R)$  pointwise everywhere. Since $\mathcal{S}(\R)$ is dense in $L^2(\R)$, it is enough to show that the four  operators  implicitly defined by the left and right-hand sides of   \eqref{potus2} and \eqref{eq:*2} are  isometric (in fact, they are unitary).

To establish  \eqref{eq:*2}, we start by noting that
$
 T_1:\mathcal{S}(\R)\rightarrow (L^2[0,{1}/{N}]^2)^d$, defined by $(T_1 f) (u,v)= S_N {P_N f_{(u,v)}},
$
is isometric. Indeed,
\[
\begin{aligned}
\|S_N {P_N f_{(u,v)}}\|^2_{ (L^2[0,{1}/{N}]^2)^d}
&=N\sum_{j=0}^{N^2-1}\iint_{[0,\frac{1}{N}]^2} \ \Bigabs{\sum_{\ell=-\infty} ^{\infty}f\Big(\frac{j}{N}+u +\ell N\Big)\e^{2\pi \im  vlN}}^2\, \dif v \dif u\\[2mm]
&=\sum_{\ell=-\infty} ^{\infty}\sum_{j=0}^{N^2-1}\int_0^{\frac{1}{N}} \Bigabs{f\Big(\frac{j}{N}+u +\ell N\Big)}^2\, \dif u %\\[2mm]
%&=\sum_{\ell=-\infty} ^{\infty}\int_0^{N}\Bigabs{f\Big(\frac{j}{N}+u +\ell N\Big)}^2\, \dif u
=\|f\|^2_{L^2(\R)}.
\end{aligned}
\]
Since the finite Fourier transform $\mathcal{F}_d$ is unitary, we conclude that $\mathcal{F}_d\circ T_1$, implicitly defined on the left hand side of \eqref{eq:*2}, is also unitary. Next, we note that $\mathcal{F}(f_{(u,v)})=e^{-2\pi i uv}(\mathcal{F}f)_{(-v,u)}$. Since both the Fourier transform $\mathcal{F}$ and the operator $T_1$ are unitary, we conclude that the operator implicitly defined on the right hand side of \eqref{eq:*2} is also unitary, and therefore that \eqref{eq:*2}  holds for all $f\in L^2(\R)$.

To obtain    \eqref{potus2} we first note that, since the discrete Zak transform is unitary, the above computation  also implies  that the operator defined by the left-hand side of \eqref{eq:*2}, which acts from $L^2(\R)$ to $(L^2[0, {1}/{N}]^2)^{N\times N}$, is isometric. To complete the proof of Proposition \ref{proposition:trump2}, it therefore remains to show that the same is true for the   operator
$
T_2:\mathcal{S}(\R)\mapsto (L^2[0,{1}/{N}]^2)^{N\times N}$, defined by $(T_2f) (u,v)= \{ Zf_{(u,v)}({m}/{N}, {n}/{N})\}_{m,n=0}^{N-1}.
$
To see this, we let  $f \in \mathcal{S}(\R)$, and make the following computation:
%{ Jan-Fredrik, I looked back at my proof for this one and did not notice a mistake, though I might be missing something. As my proof was shorter, could you look at things again? I leave your changes in color so it will be simple to note where they are}.
\[
\begin{aligned}
\Big\| \Big\{ Zf_{(u,v)}&\Big(\frac{m}{N}, \frac{n}{N}\Big)\Big\} \Big\|^2_{(L^2[0,{1}/{N}]^2)^{N\times N}}=\\[2mm]
%&=\sum_{m,n=0}^{N-1}\iint_{[0,\frac{1}{N}]^2} \Bigabs{Zf_{(u,v)}\Big(\frac{m}{N}, \frac{n}{N}\Big) }^2 \dif v \dif u\\[2mm]
&=\sum_{m,n=0}^{N-1}\iint_{[0,\frac{1}{N}]^2}  \Bigabs{\sum_{\ell=-\infty}^{\infty}f\Big(\frac{m}{N}+u-\ell \Big)  e^{-2\pi i v\ell }   e^{2\pi i\frac{n\ell}{N}}  }^2 \dif v \dif u \\[2mm]
&=\sum_{m=0}^{N-1}\int_{[0,\frac{1}{N}]}\int_{[0,1]}   \Bigabs{\sum_{\ell=-\infty}^{\infty}f\Big(\frac{m}{N}+u-\ell \Big)  e^{-2\pi i v\ell }   }^2 \dif v \dif u \\[2mm]
&= \sum_{m=0}^{N-1}\int_{[0,\frac{1}{N}]}  \sum_{\ell=-\infty}^{\infty}\Big |f\Big(\frac{m}{N}+u-\ell \Big)\Big |^2 \dif u = \norm{f}^2_{L^2(\R)}.
%&{= \sum_{m=0}^{N-1}\iint_{[0,\frac{1}{N}]^2}\sum_{n=0}^{N-1}  \Bigabs{\sum_{r=0}^{N-1} \bigg( \sum_{k=-\infty}^{\infty}f\Big(\frac{m}{N}+u-Nk - r \Big) e^{-2\pi i v Nk }   \bigg)  e^{2\pi i\frac{nr}{N}} }^2 \dif v \dif u } \\[2mm]
%&{=N \sum_{m=0}^{N-1}\iint_{[0,\frac{1}{N}]^2}\sum_{r=0}^{N-1}  \Bigabs{  \sum_{k=-\infty}^{\infty}f\Big(\frac{m}{N}+u-Nk - r \Big) e^{-2\pi i v Nk }    }^2 \dif v \dif u }.
\end{aligned}
\]
\end{proof}

In particular, we have the following.

\begin{corollary}\label{everyoneriesz}
Let $g\in L^2(\R)$ be such that the Gabor system $G(g)$ is a Riesz basis in $L^2(\R)$ with lower and upper Riesz basis bounds $A$ and $B$, respectively. Then, for almost every $(u,v)\in [0,{1}/{N}]^2$, the Gabor system $G_d(S_NP_N g_{(u,v)})$ is a Riesz basis in $\ell_2^d$ with lower and upper Riesz basis bounds $\widetilde A$ and $\widetilde B$ satisfying $A \leq \widetilde A$ and $\widetilde B \leq B$, respectively.
\end{corollary}
\begin{proof}
In light of Proposition \ref{gaboriesz}, this result follows immediately from part (i) of Proposition \ref{proposition:trump2}.
\end{proof}

\subsection{The classical Balian-Low theorem}

We start with the following lemma which relates the discrete and continuous derivatives of $L^2(\R)$ functions.

\begin{lemma}\label{theshitlemmawhichihatehatehate}
Let $f\in L^2(\R)$ and $N\in \N$. Denote $F_{(u,v)}=S_NP_Nf_{(u,v)}$. Then,
\[
N^2\|\Delta_{(j)} F_{(u,v)}(j)\|_{(L^2[0,{1}/{N}]^2)^d}^2\leq 2\int_{\R}|f'(t)|^2\dif t+ \frac{8\pi^2}{N^2}\int_\R|f(t)|^2\dif t,
\]
where the integral on the right-hand side is understood to be infinite if $f$ is not absolutely continuous, or if its derivative is not in $L^2(\R)$.
\end{lemma}
\begin{proof}
Put $a_u(j,l)=f(u+\frac{j}{N}+ \ell N)$  and fix $j\in [0,N^2-1] \cap \Z$. Since $f\in L^2(\R)$, we have, with respect to the variable $\ell$, that $a_u(j,\ell) \in \ell_2(\Z)$ for almost every $u$. We compute

\[
\begin{aligned}
\iint_{[0,\frac{1}{N}]^2}&|\Delta_{(j)} F_{(u,v)}(j)|^2\, \dif v \dif u=\\[2mm]
&=\iint_{[0,\frac{1}{N}]^2}\Big|\sum_{\ell= - \infty}^\infty \Big( a_u(j+1,l) \, \e^{2\pi i \frac{ v (j+1)}{N}}- a_u(j,l) \, \e^{2\pi i \frac{  v j}{N}}\Big) \, \e^{2\pi \im  vlN}\Big |^2\, \dif v \dif u\\[2mm]
%
%&{=\iint_{[0,\frac{1}{N}]^2}\Big|\sum_{\ell= - \infty}^\infty \Big( a_u(j+1,l) \, \e^{2\pi i \frac{ v}{N}}- a_u(j,l)   \Big) \, \e^{2\pi \im  vlN}\Big |^2\, \dif v \dif u}\\[2mm]
%
&\leq 2\iint_{[0,\frac{1}{N}]^2}\Big|\sum_{\ell= - \infty}^\infty  \Delta_{(j)} a_u(j,l) \,  \e^{2\pi \im v N \ell} \Big|^2 \dif u \dif v \\[2mm]
&\hspace{30mm}+ 2 \iint_{[0,\frac{1}{N}]^2}\Big| \Big( \e^{2\pi i \frac{  {v}}{N}} - 1\Big) \sum_{\ell= - \infty}^\infty  a_u(j,l) \,   \e^{2\pi \im  vlN}\Big |^2\, \dif v \dif u\\[2mm]
&\leq \frac{2}{N}\sum_{\ell= - \infty}^\infty\int_0^{\frac{1}{N}} | a_u(j+1,l)- a_u(j,l)  |^2\, \dif u
+ \frac{8\pi^2}{N^5}\sum_{\ell= - \infty}^\infty\int_0^{\frac{1}{N}} | a_u(j,l) |^2\, \dif u,
\end{aligned}
\]
where we applied the  inequality $|a+b|^2\leq 2|a|^2+2|b|^2$ in the second step and the inequality $|e^{ix}-1|\leq |x|$ in the third step. By the Cauchy-Schwartz inequality,  we get
\[
\begin{aligned}
\int_0^{\frac{1}{N}}\Big| a_u(j+1,l)-a_u(j,l)\Big |^2\, \dif u
&=\int_0^{\frac{1}{N}}\Big| \int_{\frac{j}{N}+ \ell N}^{\frac{j+1}{N}+ \ell N} f'(t+u)\dif t\Big |^2\, \dif u\\
&\leq\frac{1}{N}\int_0^{\frac{1}{N}}\int_{\frac{j}{N}+ \ell N}^{\frac{j+1}{N}+ \ell N} |f'(t+u)|^2\, \dif t \dif u.
\end{aligned}
\]
Combining these two estimates we find that
\begin{align*}
N^2\|&\Delta F_{(u,v)}(j)\|_{(L^2[0,{1}/{N}]^2)^d}^2=N^3\sum_{j=0}^{N^2-1}\iint_{[0,\frac{1}{N}]^2}|\Delta F_{(u,v)}(j)|^2\, \dif v \dif u\\[2mm]
&\leq 2N\sum_{\ell= - \infty}^\infty\sum_{j=0}^{N^2-1}\int_0^{\frac{1}{N}}\int_{\frac{j}{N}+ \ell N}^{\frac{j+1}{N}+ \ell N} |f'(t+u)|^2\, \dif t \dif u \\[2mm]
& \hspace{60mm}+ \frac{8\pi^2}{N^2}\sum_{\ell= - \infty}^\infty\sum_{j=0}^{N^2-1}\int_0^{\frac{1}{N}} | \alpha_u(j,l) |^2\, \dif u\\[2mm]
& =  2N\int_0^{\frac{1}{N}}\int_{\R}|f'(t)|^2\, \dif t \dif u+ \frac{8\pi^2}{N^2}\int_\R |f(t)|^2\dif t\\[2mm]
&= 2\int_{\R}|f'(t)|^2\dif t+ \frac{8\pi^2}{N^2}\int_\R |f(t)|^2\dif t.
\end{align*}
\end{proof}

We are now ready to show that the classical Balian-Low theorem    (Theorem A) follows from our finite Balian-Low theorem (Theorem \ref{finite blt}). 
\begin{proof}[Proof of the Classical Balian-Low Theorem]
Let $g\in L^2(\R)$ be such that the Gabor system $G(g)$ is a Riesz basis with lower and upper Riesz basis bounds $A$ and $B$, respectively. For all integers $N\geq 5$, $d = N^2$, and $u,v\in [0,1/N]^2$, we consider the finite dimensional signal $S_NP_N g_{(u,v)}$.   By Corollary \ref{everyoneriesz},  for almost every $u,v\in [0,1/N]^2$, this is a basis in $\ell_2^{d}$ with  Riesz basis bounds $\widetilde A$, $\widetilde B$ satisfying $A \leq \widetilde A$ and $\widetilde B \leq B$. By Theorem \ref{finite blt}, (see also Remark \ref{finite blt riesz}) we have
\[
c\log N\leq N\sum_{j=1}^{d-1}|\Delta_{(j)} S_NP_Ng_{u,v}(j)|^2+N\sum_{k=1}^{d-1}|\Delta_{(k)} \mathcal{F}_dS_NP_Ng_{u,v}(k)|^2.
\]
Integrating both parts with respect to $(u,v)$ over the set $[0, {1}/{N}]^2$, and applying Proposition \ref{proposition:trump2} (ii), we get
\[
c\log N\leq  N^2\|\Delta_{(j)} S_NP_Ng_{u,v}(j)\|^2_{(L^2[0,1/N])^d}+N^2 \|\Delta_{(k)} S_NP_N\mathcal{F}g_{u,v}(k)\|^2_{(L^2[0,1/N])^d}.
\]
By Lemma \ref{theshitlemmawhichihatehatehate}, we obtain
\begin{align*}
c\log N\leq 2\int_{\R}| g'(t)|^2\dif t+ &\frac{8\pi^2}{N^2}\int_\R| g(t)|^2\dif t \\[2mm]
&+2\int_{\R}|(\mathcal{F}g)'(t)|^2\dif t+ \frac{8\pi^2}{N^2}\int_\R|\mathcal{F} g(t)|^2\dif t.
\end{align*}
Finally,  letting $N$ tend to infinity, the result follows.
\end{proof}

\subsection{A quantiative Balian-Low theorem}

Discrete and continuous tail estimates are related by the following lemma.

 \begin{lemma}\label{thenotsoshitlemma}
Let $f\in L^2(\R)$ and $Q,N\in \N$ be such that $Q\leq N$. Denote $F_{(u,v)}=S_NP_Nf_{(u,v)}$. Then,
\[
N\sum_{j=QN}^{d-1}\|F_{(u,v)}(j)\|^2_{L^2[0,\frac{1}{N}]}\leq \int_{\R\backslash [0,Q]}|f(t)|^2\dif t.
\]
\end{lemma}
\begin{proof}
We have,
\[
\begin{aligned}
N\sum_{j=QN}^{d-1}\|F_{(u,v)}(j)\|^2_{L^2[0,\frac{1}{N}]}&=N\sum_{j=QN}^{d-1}\iint_{[0,\frac{1}{N}]^2}\Big|\sum_{\ell=-\infty}^{\infty}f\Big(u+\frac{j}{N}+ \ell N\Big)\e^{2\pi \im  v\ell N}\Big|^2\, \dif u \dif v\\[2mm]
&=\sum_{j=QN}^{d-1}\int_{[0,\frac{1}{N}]}\sum_{\ell=-\infty}^{\infty}\Big|f\Big(u+\frac{j}{N}+ \ell N\Big)\Big|^2\, \dif u\\[2mm]
&=\sum_{\ell=-\infty}^{\infty}\int_{Q}^N \Big|f(u+ \ell N)\Big|^2\, \dif u\\[2mm]
&\leq \int_{\R\backslash [0,Q]}|f(t)|^2\dif t.
\end{aligned}
\]
\end{proof}
We are now ready to show that the Quantitative Balian-Low theorem (Theorem B) follows from Theorem \ref{finite quantitative BLT 1} (or rather, the more general Theorem \ref{finite quantitative BLT 1 -rb}).  
\begin{proof}[Proof of the Quantitative Balian-Low Theorem]
Let $g\in L^2(\R)$ be such that $G(g)$ is a Riesz basis with lower and upper Riesz basis bounds $A$ and $B$, respectively, and let {$Q,R$ be positive integers}. Let $N\geq   {200} \max\{Q,R\}\cdot\sqrt{B/A}$ and set $d = N^2$. For fixed $u,v\in\R$, consider the finite dimensional signal $S_NP_Ng_{(u,v)}$.   By Corollary \ref{everyoneriesz}, for almost every $u,v\in [0,1/N]^2$, this is a basis in $\ell_2^{d}$ with  Riesz basis bounds $\widetilde A$, $\widetilde B$ satisfying $A \leq \widetilde A$ and $\widetilde B \leq B$. By Theorem \ref{finite quantitative BLT 1 -rb} we have
\[
\frac{C}{QR}\leq \frac{1}{N}\sum_{j=2QN}^{d-1}| S_NP_Ng_{u,v}(j)|^2+\frac{1}{N}\sum_{j=2RN}^{d-1}|\mathcal{F}_dS_NP_Ng_{u,v}(j)|^2.
\]
Integrating both {terms} with respect to $(u,v)$ over the set $[0,{1}/{N}]^2$, and applying Proposition \ref{proposition:trump2} (ii), we get
\[
\frac{C}{QR}\leq N\sum_{j=2QN}^{d-1}\|S_NP_Ng_{u,v}(j)\|_{(L^2[0,1])^d}^2+N\sum_{j=2RN}^{d-1}\|S_NP_N\mathcal{F}g_{u,v}(j)\|_{(L^2[0,1])^d}^2.
\]
By Lemma \ref{thenotsoshitlemma} we obtain,
\[
\frac{C}{QR}\leq \int_{\R \backslash [0,2Q]}|g(t)|^2\dif t+\int_{\R \backslash [0,2R]}|\mathcal{F}g(\xi)|^2 \dif t.
\]
The result now follows by applying an appropriate translation and modulation to $g$ (note that a translations and modulations of $g$ preserve the Riesz basis properties of $G(g)$).
\end{proof}

\section{Balian-Low theorems for Gabor systems over rectangles} \label{xcom7}

In this   section, we turn to the proofs of  theorems \ref{finite blt-rec} and \ref{finite quantitative BLT 1-rec}, which consider finite Gabor systems over rectangular lattices. Note that, as the proofs are very similar to those of theorems \ref{finite blt} and \ref{finite quantitative BLT 1}, respectively, we only provide the necessary preliminaries and  an outline of the main ideas.  For easy reference, we give the subsections below   titles that match those of  the discussion in the square lattice case.

\subsection{The finite Zak transforms over rectangles.}

Throughout this section, we let $M,N \in \N$ and denote $d=MN$. Recall from the introduction that we let $\ell_2^{(M,N)}$ denote the space of functions on the group $\Z_d = \Z/d\Z$ with norm
\begin{equation} \label{normal-rec}
	\norm{a}_{(M,N)}^2  = \frac{1}{M} \sum_{j=0}^{d-1} |a(j)|^2, \qquad a = \{ a(j) \}_{j=0}^{d-1}.
\end{equation}
With   the   normalization
\begin{equation*}
	\mathcal{F}_{(M,N)}a(k) =  \frac{1}{M} \sum_{j=0}^{d-1} a(j) \e^{-2\pi \im \frac{jk}{d}},
\end{equation*}
it is readily checked that the Fourier transform is a unitary operator from $\ell_2^{(M,N)}$ to $\ell_2^{(N,M)}$.

We   use the periodic convolution, defined for $a,b \in \ell_2^{(M,N)}$,
%\textcolor{blue}{If you agree to change one of them to phi then here it should be done too. Also, if you agree to let "b" be the star both when a generator is considered and when just a function is considered, then it includes this section as well, obviously.} {(not sure why one of these should be $\phi$ at this point... if you look at earlier sections, then at this point we use $a,b$, and in the results below, we use $a, \phi$...)}
  by
\begin{equation*}
	(a \ast b)(k) = \frac{1}{M} \sum_{j=0}^{d-1} a(k-j) b(j).
\end{equation*}
This yields  the convolution relation  $\mathcal{F}_{(M,N)}(a \ast b) = \mathcal{F}_{(M,N)} a \cdot \mathcal{F}_{(M,N)}b$. %\textcolor{blue}{I removed one line here which I think is redundant at this point.}
%Note that for the choice $c(j) = g(j/N)$, the discrete Fourier transforms and convolutions yield natural discretizations of their respective counterparts on $\R$.

We will make use of the space
$\ell_2([0,M-1]\times [0,N-1])$, with the normalization
\[
\frac{1}{d}\sum_{m=0}^{M-1}\sum_{n=0}^{N-1}|W(m,n)|^2.
\]

%{(above, and below, we use (m,n) and (M,N) as notations -- which I don't see the problem of doing -- however sometimes they come in the wrong order -- and more problematic -- we usually pair m with N, and n with M, respectively. This makes everything quite hard to read. However, it is not completely unproblematic to fix this, since we need to change the definition of the Gabor system to reverse the order of M,N)}\textcolor{blue}{Go for it}. {(Ok, but I did not do this yet)}

The finite   Zak transform for $a\in \ell_2^{(M,N)}$ is given by
	\begin{equation*}
		Z_{(M,N)}(a) \, (m,n)  =  \sum_{j=0}^{N-1}  a(m-Mj) \e^{2\pi \im \frac{jn}{N}}, \qquad (m,n) \in \Z_d^2.
	\end{equation*}
Note that with this definition, $Z_{(M,N)} (a)$ is well-defined as a function on $\Z_d^2$ (that is, it is $d$-periodic separately in each variable).
%Similar to the continuous Zak transform, the point of introducing the finite Zak transform with respect to a lattice $\Lambda = \Z_a \times \Z_b$, is that it  diagonalises Gabor systems on the   lattice in the sense that
%\begin{align*}
%	Z_\Lambda(T_{ak} M_{b\ell} c)(m,n) = \e^{2\pi \im \frac{b \ell  m - ak n }{d}} Z_\Lambda c\,(m,n).
%\end{align*}
%
%{Riesz basis property...}

%Below, we state some  basic properties of the finite Zak transform, which are all analog to corresponding properties of  the Zak transform on $L^2(\R)$ (which we will refer to as the continuous case). We omit the proofs since they all follow closely those of the continuous counter-part described in the standard reference   \cite{grochenig}. {(Here, I am thinking of Grˆchenig, but Shahaf mentioned a Heil paper directly on the finite guy.)}
The basic properties of $Z_{(M,N)}$ are stated in the following lemma  (compare with Lemma \ref{ztprop...}).
\begin{lemma} \hspace{0mm} \label{lemma:basics on finite Zak-rec} Let $M,N\in \N$ and $d=MN$.   For $a \in  \ell_2^{(M,N)}$,  the following hold.
\begin{itemize}
	\item[(i)] The function $Z_{(M,N)}(a)$   is $(M,N)$-quasi-periodic on $\Z_d^2$ in the sense that
\begin{equation} \label{cond:qp2-rec}
\begin{aligned}
	Z_{(M,N)}(a)(m + M,n) &= \e^{2\pi \im \frac{n}{N}}Z_{(M,N)}(a)(m,n),  \\[2mm]
	Z_{(M,N)}(a)(m,n+N) &=  Z_{(M,N)} (a)(m,n).
\end{aligned}
\end{equation}
 In particular, the finite Zak transform is determined by its values on the set $\Z^2 \cap ([0,M-1]\times [0,N-1])$.
	\item[(ii)] The   transform $Z_{(M,N)}$ is a unitary operator from $\ell_2^{(M,N)}$ onto $\ell_2([0,M-1]\times [0,N-1])$.
	\item[(iii)]  The finite Zak transform and the finite Fourier transform satisfy the relation
\begin{equation} \label{Zak and Fourier-rec}
	Z_{(N,M)} (\mathcal{F}_{(M,N)}a)(n,m) =   \, \e^{2\pi \im \frac{mn}{d}} Z_{(M,N)}(a) (-m,n).
\end{equation}
	\item[(iv)] For $a,\phi \in  \ell_2^{(M,N)}$ the finite Zak transform satisfies the convolution relation
\begin{equation*}
	Z_{(M,N)} (a \ast \phi) = Z_{(M,N)}(a) \ast_1 \phi,
\end{equation*}
where the subscript of $\ast_1$ indicates that the  convolution is taken with respect to the first variable of the finite Zak transform.
\end{itemize}
\end{lemma}

An extension of part (ii) of Proposition \ref{gaboriesz} to the rectangular case reads as follows.
\begin{proposition} \label{gaboriesz-rec} % \cite{grochenig2001,???}
 Let $M,N\in \N$ and $b\in \ell_2^{(M,N)}$. Then the Gabor system $G_{(M,N)}(b)$,   defined in \eqref{discrete gabor rect}, is a basis in $\ell_2^{(M,N)}$ with Riesz basis bounds $A$ and $B$, if and only if $A\leq |Z_{(M,N)}(b)(m,n)|^2\leq B$ for all $(m,n) \in \Z^2\cap([0,M-1]\times [0,N-1])$.
\end{proposition}

\subsection{Relating continuous and finite signals in the rectangular case}

For $M, N \in \mathbb{N}$, $d=MN$, $f \in \mathcal{S}(\R)$, and an $N$-periodic function $h$ over $\R$, recall the notations
\begin{equation*}
	P_{N} {f}(x)  = \sum_{\ell= - \infty}^\infty f(x+ \ell N)
\qquad and
\qquad
	S_{M} \, h  = \{ h(j/M) \}_{j=0}^{d-1}.
\end{equation*}

The following result extends  proposition \ref{proposition:trump1}.

\begin{proposition}\label{proposition:trump1-rec}
	\hspace{2mm}
For $f \in \mathcal{S}(\R)$,   the following hold.
	\begin{itemize}
\item[(i)]   For every $M,N\in\N$ and $(m,n)\in\Z^2$, we have
		\begin{equation} \label{potus-rec}
			Z_{(M,N)} (S_{M} {P_{N} f})(m,n) = Zf(m/M, n/N).
		\end{equation}
		\item[(ii)] For every $M,N\in\N$, we have
		\begin{equation} \label{eq:*-rec}
			  \mathcal{F}_{(M,N)}S_{M} P_{N} {f} = S_{N} P_{M} \mathcal{F} f.
		\end{equation}
%		where $\mathcal{F}_\R$ is the $L^2(\R)$ Fourier transform.
	\end{itemize}
\end{proposition}

 We now point out that the estimate of   Lemma \ref{derder}  can be extended to the rectangular setting. Indeed, for
 $M,N\in\N$, $d=MN$ and $f\in \mathcal{S}(\R)$,  denote $F=S_{M}P_{N}f$. Then,
\[
\sum_{j=0}^{d-1}|\Delta F_j|  \leq \int_{\R}|f'(x)|\dif x.
\]

\subsection{Regularity of the finite Zak transforms over rectangles}

Lemma \ref{lemma:jump 1} is already   formulated in a rectangular version. One can obtain a similar analog of Lemma \ref{douche} with the notations
\begin{equation} \label{sigmaomega-rec}
		\sigma_s = \Big[ \frac{s M}{K} \Big], \quad s \in [0,K]\cap \Z,\qquad \omega_t = \Big[ \frac{t N}{L} \Big], \quad t \in [0,L]\cap \Z,
\end{equation}
for integers  $K \in [2,M]$ and $L \in [2,N]$. Note that, similar to the square case, we have $\sigma_K=M$ and $\omega_L=N$.
With this we get the following extension of Corollary \ref{lem:infidel1}.
\begin{corollary} \label{lem:infidel1-rec}
Fix an integer $N_0 \geq 5$ and a constant $A>0$. Let
$$\delta=2\sqrt{A}\sin\Big( \pi \Big(\frac{1}{4}-\frac{1}{N_0}\Big)\Big).$$
For any integers
	$M,N \geq N_0$, $K\in [2,M]$, and $L\in [2,N]$, the following holds (with $d=MN$).
	Let $W$ be an $(M,N)$-quasi periodic function over the lattice $\Z_d^2$ satisfying $A \leq \abs{W(m,n)}^2$, for all $(m,n)\in\Z^2_d$.
	Then,  for every $(u, v) \in \Z_d^2$ there exists at least one point $(s,t)\in ([0,K-1]  \times [0,L-1]) \cap \Z^2$ such that
	\begin{align}
			&\abs{\Delta_{(s)} W(u+\sigma_{s}  , v+\omega_t)} \geq \delta, \quad \text{or} \label{baloney10-rec} \\[2mm]
			&\abs{\Gamma_{(t)} W(u+\sigma_s, v+\omega_{t}) } \geq \delta,   \label{baloney11-rec}
	\end{align}
where  $\sigma_s, \omega_t$ are   defined in (\ref{sigmaomega-rec}).
	\end{corollary}

%	{(From what I can see, this is taken care of the $N_0$ in the expression for $\delta$... indeed, compare with Corollary \ref{lem:infidel1}. If you want, we can use different deltas in the two directions, but I don't think it is worth the trouble.) }

\subsection{A   Balian-Low type theorem in finite dimensions over rectangles}
%\subsection{An equivalent measure of {smoothness} for finite sequences}

Fix $M,N\in\N$ and set $d=MN$. For $b\in \ell^{(M,N)}_2$, denote
\[
\alpha(b,M,N):=M\sum_{j=0}^{d-1} \abs{\Delta b(j)}^2  +  N\sum_{k=0}^{d-1} \abs{\Delta \mathcal{F}_{(M,N)}{b}(k)}^2,
\]
and
\[
\beta(b,M,N):=\frac{M}{N}\sum_{m=0}^{M-1} \sum_{n=0}^{N-1} \abs{\Delta Z_{(M,N)}(b)\,(m,n)}^2 + \frac{N}{M}\sum_{m=0}^{M-1} \sum_{n=0}^{N-1}  \abs{\Gamma Z_{(M,N)} (b)\,(m,n)}^2.
\]
Given $A,B>0$ set
\[
\alpha_{A,B}(M,N)=\inf\{\alpha(b,M,N)\}\qquad\textrm{ and }\qquad\beta_{A,B}(M,N)=\inf\{\beta(b,M,N)\},
\]
where both infimums are taken over all   $b\in \ell^{(M,N)}_2$  for which the system $G_{(M,N)}(b)$ is a basis with lower and upper Riesz basis bounds $A$ and $B$, respectively.
The analog of Proposition \ref{prop:alpha and beta} for the rectangular case gives the inequalities
	\begin{equation}\label{alphabeta-rec}
		 \frac{1}{2}\beta_{A,B}(M,N) - 8\pi^2 B   \leq \alpha_{A,B}(M,N) \leq  2\beta_{A,B}(M,N) +  8\pi^2 B.
	\end{equation}
It follows that Theorem \ref{finite blt-rec}  is a consequence of the following result.
\begin{theorem}\label{finite blt with betta-rec}
	There exist constants $c,C>0$ so that,  for all integers $M,N \geq 2$, we have
\begin{equation}\label{thiswillneverend}
c \log \min\{M,N\}\leq \beta_{A,B}(M,N)\leq  C \log \min\{M,N\}.
\end{equation}
\end{theorem}

\begin{proof}
In what follows we let $c, C$ denote different constants which may change from line to line.
	Let $b\in \ell^{(M,N)}_2$  be so that $G_{(M,N)}(b)$ is a basis with Riesz basis bounds $A$ and $B$, and put $W = Z_{(M,N)}(b)$. To obtain the lower inequality, we first consider  the case   $M>N$. Put $\sigma_s = [s M/N]$, and  consider square product sets
	%\textcolor{blue}{as usual I point out that these are not sublattices, maybe we should call them quasi-sublattices, I don't know?} {("derived lattices"?)}
	\begin{equation*}
		\mathrm{Lat}_k =  \{ ( k + \sigma_s, t) :  s,t {\in [0,N-1]\cap Z} \}.
	\end{equation*}	
	Note that  similar   to \eqref{infsup}, it holds that
	\begin{equation}\label{never ever}
		\Big[ \frac{M}{N} \Big] \leq \inf \Delta \sigma_s \leq \sup \Delta \sigma_s \leq \Big[ \frac{M}{N} \Big] +1. %{\leq 2 \Big[ \frac{M}{N} \Big].}
	\end{equation}
	This implies that, for integers $0 \leq k < [M/N]$, the $N\times N$ squares $\mathrm{Lat_k}$ are disjoint as subsets of the $M \times N$
rectangle  $([0,M-1]  \times [0,N-1])\cap \Z^2$. Since, when restricted to  each of the  square sets
	 $\{ ( k + \sigma_s, t) :   s,t {\in [0,N]\cap Z} \},$
	the function $W$ is  N--quasi--periodic, it follows by Theorem \ref{finite blt with betta} that
	\begin{equation}\label{thiswillneverend}
		c \log N \leq \sum_{s,t=0}^{N-1} \Big( \abs{\Delta_{(s)} W(k+\sigma_s,t)}^2
		+   \abs{\Gamma_{(t)} W(k+\sigma_s,t)}^2\Big).
	\end{equation}
By \eqref{never ever}, and the Cauchy-Schwarz inequality, we have
\[
\begin{aligned}
 \sum_{s=0}^{N-1}\abs{\Delta_{(s)} W(k+\sigma_s,t)}^2&=  \sum_{s=0}^{N-1}\Bigabs{\sum_{j=k+\sigma_s}^{k+\sigma_{s+1}  -1 }\Delta_{(j)} W(j,t)}^2\\
 &\leq  \Big(\Big[\frac{M}{N}\Big] +1\Big)\sum_{s=0}^{N-1}\sum_{j=k+\sigma_s}^{k+\sigma_{s+1} -1}\abs{\Delta_{(j)} W(j,t)}^2\\
  &\leq  \Big(\Big[\frac{M}{N}\Big] +1\Big)\sum_{j=0}^{M-1}\abs{\Delta_{(j)} W(j,t)}^2,
 \end{aligned}
 \]
where the last step of the above computation follows from the fact that the summands are $M$-periodic. Plugging this into \eqref{thiswillneverend} we get,
 \[
		 c \log N \leq \Big[\frac{M}{N}\Big]\sum_{t=0}^{N-1}\sum_{j=0}^{M-1}\abs{\Delta_{(j)} W(j,t)}^2
		+  \sum_{s,t=0}^{N-1} \abs{\Gamma_{(t)} W(k+\sigma_s,t)}^2\Big..
	\]
	Since each of the  $ [{M}/{N} ]$ sets $\mathrm{Lat}_k$ are disjoint, we may sum up these inequalities  to obtain
	\begin{equation*}
	c\Big[\frac{M}{N}\Big]\log N\leq \Big[\frac{M}{N}\Big]^2\sum_{m=0}^{M-1}\sum_{n=0}^{N-1}  \abs{\Delta Z_{(M,N)} (b)\,(m,n)}^2 +\sum_{m=0}^{M-1}\sum_{n=0}^{N-1}  \abs{\Gamma Z_{(M,N)} (b)\,(m,n)}^2.
	\end{equation*}
	Dividing by $[M/N]$ on both sides, and using the inequalities $x/2 \leq [x] \leq x$ valid for $x \geq1$, it follows that
	\begin{equation} \label{rec-tam}
		c \log N \leq \beta(b,M,N).
	\end{equation}
	By repeating the same type of argument for the case $M<N$, the estimate from below in \eqref{thiswillneverend} follows.

	To obtain the upper inequality in \eqref{thiswillneverend}, we consider the same function $G(x,y)$ that was used in the proof for Theorem \ref{finite blt}, and split into two cases as above.

	We begin with the case $M>N$, where it suffices to go through essentially the same computations as in the proof of Proposition \ref{upperten},   this time considering a vector   $b \in \ell_2^{(M,N)}$ so that $Z_{(M,N)}(b)(m,n) = G(m/M,n/N)$, where $G$ is the same function used in square case. With this, we obtain the inequality
%	 square lattice $\{\textcolor{magenta}{(m/N, n/N)} : 0 \leq m,n \leq N\}$  by the  rectangular lattice $\{\textcolor{magenta}{(m/M, n/N)} : 0 \leq m < M, 0 \leq n < N\}$.  In the notations above, this yields a vector $b\in \ell_2^{(M,N)}$ which satisfies the inequality
	\begin{equation} \label{banjo}
		\beta(b,M,N) \leq C \log N,
	\end{equation}
	for some constant $C>0$.

	We turn to the   case   $M<N$.    Swapping the roles of $m,M$ and $n,N$, respectively,  it follows that     the vector $b\in\ell_2^{(N,M)}$ which satisfies $Z_{(N,M)}b(n,m)=G(n/N,m/M)$  is the same as in the above case, whence  $\beta(b,N,M) \leq C \log M$. Now, since $\alpha(b,N,M)=\alpha(\mathcal{F}_{(N,M)}b,M,N)$, the relation \eqref{alphabeta-rec} implies that 
\begin{equation} \label{kazooie}
\beta(\mathcal{F}_{(N,M)}b,M,N) \leq C\beta(b,N,M) \leq C  \log M.
\end{equation}
	In combination, the inequalities \eqref{banjo} and \eqref{kazooie} yield the desired upper bound for $\beta(M,N)$.
\end{proof}

\subsection{A quantitative Balian-Low type theorem in finite dimensions over rectangles}

The following lemma is an extension of Lemma \ref{conv-lemma}.
\begin{lemma}\label{conv-lemma-rec}
Let $A, B>0$ and $M,N\geq    {200}\sqrt{B/A}$. There exist positive constants $\delta=\delta(A)$ and $C=C(A,B)$ such that   the following holds (with $d=MN$).  Given, \begin{itemize}
\item[(i)]   $\quad Q, R$ such that $1\leq Q   < (N/{16}) \cdot  \sqrt{A/B}$ and $1\leq R   < (M/{16}) \cdot  \sqrt{A/B}$,
\item[(ii)]  $\quad \phi \in \ell_2^{(M,N)}$ such that $\sum_n|\Delta\phi(n)|\leq {10} R$,
\item[(iii)]   $\quad \psi \in \ell_2^{(N,M)}$  such that $\sum_n|\Delta\psi(n)|\leq {10}   Q$,
\item[(iv)]  $\quad b\in \ell_2^{(M,N)}$ such that $A \leq |Z_{ (M,N)}(b)|^2  \leq B$,
\end{itemize}
then there exists a set $S\subset ([0,M-1]\times [0,N-1])\cap\Z^2$ of size $|S|\geq C MN/ QR$
 such that all $(u,v)\in S$ satisfy either
\begin{align}\label{conv-ineq-1-rec}
|Z_{(M,N)}(b)(u,v)-Z_{(M,N)}(b\ast \phi)(u,v)|\geq\delta, \qquad \text{or} \\[2mm]
\label{conv-ineq-2-rec}
|Z_{(N,M)}( \mathcal{F}_{(M,N)}  b) (v,u) -Z_{(N,M)}(( \mathcal{F}_{(M,N)} b)\ast \psi) (v,u)|\geq\delta.
\end{align}
\end{lemma}

%{(Hmmm... I just noticed that it seems strange that the same $(u,v)$ are used as inputs for both $Z_{(M,N)}$ and $Z_{(N,M)}$ since these guys operate on different lattices... this should be checked by someone at some time)}\textcolor{blue}{I think that now it's OK, note the blue change. BTW, recall that we have a rectangle file one can go to in order to see where the typo is.}

The proof of the above lemma is, essentially word by word, the same as the one for Lemma \ref{conv-lemma}.  Indeed, the same  argument works on a rectangular lattice as soon as  one makes the  choices $\sigma_s = [sM/K]$ and $\omega_t = [tN/L]$ for $s,t \in \Z$.

Similarly, the proof of the   quantitative Balian--Low theorem in finite dimensions over finite rectangular lattices  (Theorem  \ref{finite quantitative BLT 1-rec}) is, essentially word by word, the same   as the one for the  finite square lattice case  (Theorem \ref{finite quantitative BLT 1 -rb}). The only changes needed in the proof are the choices $\phi = S_{M} P_{N} \Phi$ and $\psi = S_{N} P_{M} \Psi$.

\bibliographystyle{amsplain}
%\bibliography{../../../@bibliotek/bibliotek}

\def\cprime{$'$} \def\cprime{$'$} \def\cprime{$'$} \def\cprime{$'$}
\providecommand{\bysame}{\leavevmode\hbox to3em{\hrulefill}\thinspace}
\providecommand{\MR}{\relax\ifhmode\unskip\space\fi MR }
\providecommand{\MRhref}[2]{%
  \href{http://www.ams.org/mathscinet-getitem?mr=#1}{#2}
}
\providecommand{\href}[2]{#2}

\end{document}